\documentclass[11pt,a4paper]{article}
\usepackage[margin=1in]{geometry}
\usepackage{graphicx}
\usepackage{amsmath, amsthm, amssymb, amsfonts}
\usepackage{bbm}
\usepackage{enumerate}
\usepackage{mathtools}
\usepackage{xcolor}
\usepackage{hyperref}
\usepackage{centernot}

\renewcommand{\epsilon}{\varepsilon}
\graphicspath{ {./images/} }
\numberwithin{equation}{section}

\theoremstyle{definition}
\newtheorem{definition}{Definition}[section]

\newtheorem{remark}[definition]{Remark}

\theoremstyle{plain}
\newtheorem{theorem}[definition]{Theorem}
\newtheorem{proposition}[definition]{Proposition}
\newtheorem{lemma}[definition]{Lemma}
\newtheorem{corollary}[definition]{Corollary}

\newcommand{\sfa}{\mathsf{a}}
\newcommand{\sfb}{\mathsf{b}}
\newcommand{\bfx}{\mathbf{x}}

\usepackage{authblk}

\title{Subcritical sharpness for real-valued spin models}

\author[1]{Christoforos Panagiotis}
\author[2]{William Veitch}
\affil[1,2]{{Department of Mathematical Sciences}\\
        {University of Bath}\\
        {UK}}

\begin{document}

\maketitle

\begin{abstract}
In this paper, we consider a large family of real-valued spin models on general transitive graphs. We show that, in the subcritical regime $\beta<\beta_c$, the correlations of the model decay exponentially fast. To prove this result, we consider the random cluster (a.k.a.~FK percolation) representation of the model and obtain an inequality that generalises the OSSS inequality to monotonic measures with random connection probabilities, thus extending the inequality of Duminil-Copin, Raoufi, and Tassion \cite{sharpness_annals}. Our results apply in particular to the Blume--Capel model and general $P(\varphi)$ models, going beyond the cases of the Ising and $\varphi^4$ models treated by Aizenman, Barsky, and Fernández \cite{ABF87}.
\end{abstract}

\section{Introduction}

In this paper, we consider a general family of spin models with a $\mathbb{Z}/2\mathbb{Z}$ symmetry, where each spin takes values in $\mathbb{R}$. The spin model on a finite graph $(\Lambda, E)$ with free boundary conditions is a probability measure on configurations $\varphi \in \mathbb{R}^{\Lambda}$, with the expectation of a bounded measurable function $f: \mathbb{R}^{\Lambda} \to \mathbb{R}$ given by
\begin{equation*}
    \langle f \rangle = \frac{1}{Z} \int_{\mathbb{R}^{\Lambda}} f(\varphi) \exp\left(\sum_{\{x,y\} \in E} \beta \varphi_x \varphi_y\right) \prod_{x \in \Lambda} \mathrm{d} \rho(\varphi_x),
\end{equation*}
where $\beta \geq 0$ is the \textit{inverse temperature}, $\rho$ is an even single-site measure, and $Z$ is the appropriate normalisation constant, called the \textit{partition function}. We are interested in the case where the single-site measure $\rho$ has super-Gaussian tails, so that the measure is well-defined for all values of $\beta\geq 0$. 

The study of this class of spin systems dates back to the 1920s, beginning with the introduction of the Ising model by Lenz and Ising \cite{Lenz1920beitrag, Ising1924} as a model of ferromagnetism. This is possibly the simplest model in our class, with spins taking the values $+1$ and $-1$, corresponding to the two possible spin orientations. Since these early days, the study of spin models has expanded drastically, with several new models appearing over the years. In the 1960s, Blume \cite{BL} and Capel \cite{Capel} independently introduced what is now called the Blume--Capel model to study phase transitions in systems of triplet ions, such as the exotic phase transition observed in the magnetisation of Uranium dioxide, by incorporating vacancies, i.e.\ allowing spins to additionally take the value $0$. Around the same time, developments in the study of Euclidean quantum field theory \cite{N66,GJ73,frohlich1977pure, borgs1989unified} led to the introduction of continuous-spin models such as the $\varphi^4$ model and, more generally, $P(\varphi)$ models. These models are defined by letting 
\[\mathrm{d}\rho(\varphi_x)=\exp(-P(\varphi_x))\mathrm{d}\varphi_x,\]
where $P$ is an even polynomial of degree at least $4$ and of positive leading coefficient. The $\varphi^4$ model is of particular interest among this class, as it is closely related to the Ising model --- see \cite{GS}. For a comprehensive introduction to these models, the interested reader can consult \cite{DC_lecturenotes,friedli_velenik_2017,Glimm-Jaffe,WP21}.

Despite the wide variety of spin models considered in this paper, they all share a common feature: they undergo a \emph{phase transition}. To define this precisely, let us first consider an infinite, connected, locally finite, transitive graph $G=(V,E)$, which is fixed throughout the paper, and let $o \in V$ be a fixed origin. We are interested in the set of infinite-volume Gibbs measures, which are defined via the \textit{Dobrushin--Lanford--Ruelle (DLR) equation} (see Definition \ref{def: Gibbs} below). Such measures arise naturally as weak limits of finite-volume measures with \emph{boundary conditions}. In cases where the spins take values in a bounded set, such as the Ising model, one can define a maximal Gibbs measure at infinite volume, called the \emph{plus measure} and denoted $\langle \cdot \rangle^+_{\beta}$, which arises as the limit of measures with maximal boundary conditions. However, when the values of the spins are unbounded, there is no canonical way of defining maximal boundary conditions. As it turns out, we can still define a plus measure which is maximal at infinite volume, provided that we restrict to the set of \emph{regular} Gibbs measures, which roughly speaking are measures having the same type of tails as the single-site measure $\rho$. This was established on the lattice $\mathbb{Z}^{d}$ by Lebowitz and Presutti \cite{Lebowitz_Presutti} utilising the regularity estimates developed by Ruelle \cite{Ruelle1970, Ruelle_estimates}. It was very recently extended to general graphs in \cite{PV26}, where it was also shown that the set of regular Gibbs measures includes, in particular, weak limits of finite-volume measures with boundary conditions growing at most exponentially fast. With the plus measure $\langle \cdot \rangle^+_{\beta}$ at hand, we define the critical point $\beta_c$ in terms of the \emph{spontaneous magnetisation} $\langle \varphi_o \rangle^+_{\beta}$ as
\[
\beta_c:=\inf\{\beta\geq 0: \, \langle \varphi_o \rangle^+_{\beta}>0\}.
\]
It is classical that when $\langle \varphi_o \rangle^+_{\beta}=0$ there exists a unique infinite-volume (regular) Gibbs measure --- see e.g.\ \cite[Proposition C.2]{random_tangled}.

In this article, we are interested in the \emph{subcritical regime} $\beta<\beta_c$. For the Ising model, it is standard that on transitive graphs one has $\beta_c>0$, so the subcritical interval is always non-empty. In Proposition~\ref{prop:non-trivial}, we show that this result extends to all models considered in this article, using the regularity results from \cite{PV26}. For the \emph{supercritical regime} $\beta>\beta_c$, the question of non-triviality is more subtle. On the one hand, one has $\beta_c=\infty$ on the one-dimensional integer lattice $\mathbb{Z}$. On the other hand, for the Ising model on the hypercubic lattice $\mathbb{Z}^d$ with $d\geq 2$, a celebrated work of Peierls \cite{P36} showed that $\beta_c<\infty$; see also \cite{GJS75,FrohlichSimonSpencerIRBounds1976} for an extension to other spin models. It is now well understood that the phase transition of the Ising model on a graph $G$ is non-trivial whenever the phase transition of Bernoulli bond percolation on $G$ is non-trivial. In Proposition~\ref{prop:non-trivial}, we extend this result to the class of real-valued spin models. Let us also mention that the question of non-triviality of the phase transition for Bernoulli percolation was recently settled in \cite{DGRSY20,EST25}, where it was shown that on transitive graphs that are not one-dimensional, the phase transition is non-trivial.

It is classical that for sufficiently small values of $\beta$, correlations decay exponentially fast; this follows, for instance, from cluster expansion methods or high-temperature expansions \cite{Glimm-Jaffe}. Our main theorem shows that this exponential decay persists throughout the entire subcritical regime, a phenomenon known as \emph{subcritical sharpness}. Here, we state a stronger result that applies also to finite-volume measures with weakly growing boundary conditions. We denote the graph distance in $G$ by $d_G$ and for $x \in V, k \geq 0$ let $\Lambda_k = \{x \in V : d_G(o, x) \leq k\}$. 

\begin{theorem}\label{thm: main theorem}
For every $\beta<\beta_c$, there exist $\lambda,c>0$ such that the following holds. For every $x\in V$ such that $d_G(o,x)$ is large enough, every $n \geq d_G(o,x)$, and all boundary conditions $\eta$ satisfying $|\eta_y|\leq e^{\lambda d_G(o, y)}$ for every $y\in V$,
\[
\langle \varphi_o \varphi_x \rangle^{\eta}_{\Lambda_n,\beta}\leq e^{-cd_G(o,x)}. 
\]
In particular,
\[
\langle \varphi_o \varphi_x \rangle^+_{\beta}\leq e^{-cd_G(o,x)}. 
\]
\end{theorem}

Theorem~\ref{thm: main theorem} extends the work of Aizenman, Barsky, and Fernández \cite{ABF87} who proved subcritical sharpness for the Ising and $\varphi^4$ models, as well as the more recent subcritical sharpness result for the Blume--Capel model \cite{Blume-Capel}. See also \cite{DuminilTassionNewProofSharpness2016} for a short proof in the case of the Ising model.

The proof of Theorem~\ref{thm: main theorem} uses the random cluster (a.k.a.~FK percolation) representation studied in \cite{well_behaved}, which is a measure on pairs $(\sfa,\omega)$, where $\sfa$ is the absolute value field and $\omega$ is a percolation configuration. Given a spin configuration $\varphi$, we sample a pair $(\sfa, \omega)$ by setting $\sfa_x = |\varphi_x|$ and noting that the sign field $\mathrm{sgn}(\varphi)$ is distributed as an Ising model with random coupling constants determined by the absolute value field $\sfa$. We then sample $\omega$ using the Edwards--Sokal coupling for the Ising model, which allows us to express the two-point functions of the spin model in terms of connectivity probabilities in the random cluster representation. The desired result then follows from proving subcritical sharpness for the random cluster representation.

For percolation models, results on subcritical sharpness go back to the 1980s \cite{mensikov1986coincidence,aizenman1987sharpness}, where they were first established for Bernoulli percolation. Since then, several new proofs of this result have appeared \cite{DuminilTassionNewProofSharpness2016, sharpness_annals, Hut20, Van23}. A common feature of these proofs is the use of differential inequalities relating the derivative of a well-chosen observable with respect to $\beta$ (or another natural parameter of the model) to a function of the observable, with the exception of \cite{Van23}, which instead relies on couplings.

The proof of Theorem~\ref{thm: main theorem} follows a strategy similar to \cite{sharpness_annals}, which applies beyond Bernoulli percolation to the classical random cluster model. At the heart of our argument is a generalisation of the OSSS inequality of O’Donnell, Saks, Schramm, and Servedio \cite{OSSS} originally introduced for product measures. This inequality uses an exploration to get an improvement to the Poincaré-type inequality that bounds the variance of an increasing function $f$ by the sum over edges $e$ of the covariances with respect to $\omega_e$. The exploration is encoded in terms of a decision tree, and the improvement arises from the revealment of an edge $e$, which is roughly speaking the probability that $e$ is explored by the decision tree.
The OSSS inequality was extended to monotonic measures in \cite{sharpness_annals}. The standard random cluster measure is indeed a monotonic measure. However, when the coupling constants themselves become random, the resulting random cluster measure is not necessarily monotonic anymore, as was already observed for the representation of the Blume--Capel model in \cite{GG}. In \cite{Blume-Capel}, a weaker property, called weak monotonicity, was introduced, which allowed for extending the OSSS inequality to the random cluster model on a site percolation cluster. In this paper, we generalise the latter result to a large family of random cluster models with random coupling constants.

Below, a decision tree $T$ is said to be admissible if it explores the state of an edge only after exploring both endpoints, and $\delta_{xy}(f, T)$ is a combination of the revealment probability of the edge $xy$ from the classical OSSS inequality with additional terms relating to the exploration of the endpoints $x$ and $y$ --- see Section \ref{section: OSSS} for the precise definitions. 

\begin{theorem}[OSSS inequality]
    \label{Thm: OSSS}
    Let $\Lambda\subset V$ be finite and let $\overline{E}(\Lambda) \subset E$ be the set of edges with at least one endpoint in $\Lambda$. Let $\mu$ be a measure on $(\sfa,\omega)\in (\mathbb{R}^+)^{\Lambda}\times \{0,1\}^{\overline{E}(\Lambda)}$ that satisfies weak monotonicity (Definition \ref{Def: weak monotonicity}) and
    for any vertex $z \in \Lambda$ and any increasing function $f: \{0,1\}^{\overline{E}(\Lambda)} \rightarrow [0, 1]$ depending only on the states of edges,
    \begin{equation}\label{eq: zero conditioning}
        \mu [f(\omega) \mid \omega_{xz} = 0 \; \forall x \sim z]
        \leq \mu [f(\omega) \mid  \mathsf{a}_z = 0].
    \end{equation}
    Then for any increasing function $f: \{0, 1\}^{\overline{E}(\Lambda)} \rightarrow [0,1]$ depending only on the states of edges and any admissible decision tree $T$,
    \begin{align*}
        \mathrm{Var}_{\mu}(f) \leq \sum_{xy \in \overline{E}(\Lambda)} \delta_{xy}(f, T) \mathrm{Cov}_{\mu} (f, \omega_{xy}).
    \end{align*}
\end{theorem}

Let us remark that in Theorem~\ref{Thm: OSSS}, despite the fact that $\mu$ is a measure on pairs $(\sfa,\omega)$, the inequality involves covariances only with respect to the edge states. This is in contrast to the inequality obtained in \cite{Blume-Capel} that involves covariances with respect to the absolute value field as well. Nevertheless, the absolute value field still plays a role through the revealment probability $\delta_{xy}(f, T)$, which is natural to expect due to the fact that in our context, decision trees explore the states of vertices as well as edges. In the proof of Theorem~\ref{Thm: OSSS}, we handle the contribution of the absolute value field using a correlation inequality, namely our assumption \eqref{eq: zero conditioning} in the statement of Theorem~\ref{Thm: OSSS}, which informally speaking states that closing a vertex $x$ is less costly than closing all edges incident to $x$. The fact that the random cluster representation of the spin models studied in this article satisfies \eqref{eq: zero conditioning} is proved in Lemma~\ref{Lemma: zero conditioning}.

With Theorem~\ref{Thm: OSSS} at hand, the first step in proving Theorem~\ref{thm: main theorem} is to apply the OSSS inequality for $f$ being the indicator of the event that $o$ is connected to the boundary of $\Lambda_n$ and certain well-chosen decision trees $T$. The next step is to interpret the resulting bound as a differential inequality. At this point, it might seem that we have reached an obstacle. Indeed, for the classical random cluster measure, the derivative of $\mu(f)$ with respect to $\beta$ is proportional to $\sum_{xy} \mathrm{Cov}_{\mu} (f, \omega_{xy})$; however, this is no longer true for random cluster measures with random coupling constants. In \cite{Blume-Capel}, this was overcome by introducing a slightly different model for which the derivative is roughly equal to $\sum_{xy} \mathrm{Cov}_{\mu} (f, \omega_{xy})$, and the edge marginal of this new model was then compared to that of the original model. However, such an approach seems specific to the random environment arising when working with the Blume--Capel model. Our approach instead relies on the correlation inequality of Lemma~\ref{lem:cov comparison new} that allows us to directly compare the derivative of $\mu(f)$ with $\sum_{xy} \mathrm{Cov}_{\mu}(f,\omega_{xy})$, thereby yielding a differential inequality from which sharpness of the phase transition can be deduced.

\subsection{Paper organisation}

In Section~\ref{sec:preliminaries}, we define the spin model on finite graphs and in infinite volume, and state its regularity properties together with several correlation inequalities. In Section~\ref{section: random cluster}, we introduce the random cluster representation and prove some correlation inequalities that will be useful throughout the paper. We then prove the non-triviality of the phase transition in Section~\ref{sec:non-triviality}. We next turn to the proof of Theorem~\ref{thm: main theorem}. We first establish the OSSS inequality in Section~\ref{section: OSSS}. We then prove subcritical sharpness for the random cluster representation in Section~\ref{sec:sharpness random cluster}. Finally, in Section~\ref{sec:sharpness}, we complete the proof of Theorem~\ref{thm: main theorem}.

\paragraph{Acknowledgements} We would like to thank Trishen Gunaratnam, Dmitrii Krachun, Romain Panis and Franco Severo for useful discussions. CP was supported by an EPSRC New Investigator Award (UKRI1019).

\section{Preliminaries}\label{sec:preliminaries}

In this section we introduce the spin model and collect the regularity and correlation inequalities that will be used throughout the paper.

\subsection{Definition of the spin model}\label{section: definitions}
Recall that $G=(V, E)$ is an infinite, connected, locally finite, transitive graph and that $d_G$ denotes the graph distance in $G$. We will write elements of $E$ in the form $xy$ (though we still consider them as undirected edges) and denote by $D$ the degree of each vertex in $G$. For $x \in V, k \geq 0$ let $\Lambda_k(x) = \{y \in V : d_G(x, y) \leq k\}$, and write $\Lambda_k = \Lambda_k(o)$.

We say that a Borel measure $\rho$ on $\mathbb R$ is an \emph{admissible single-site measure} if it satisfies the following properties:
\begin{enumerate}
    \item[$(i)$] $\rho$ is \emph{even}, which means that for every Borel measurable set $A$, $\rho(A)=\rho(-A)$. 
    \item[$(ii)$] $\rho$ has \emph{super-Gaussian} tails, in the sense that
    \begin{equation*}
    \forall  a > 0 \quad \int_{\mathbb{R}} e^{a |u|^{2}} \mathrm{d} \rho(u) < \infty.
    \end{equation*}
    \item[$(iii)$] $\rho$ has non-trivial support, i.e.\ $\rho(\mathbb{R}\setminus\{0\})>0$.
\end{enumerate}
Assumption (ii) above is necessary for the model to be well-defined for all $\beta \geq 0$ due to the quadratic nature of the interactions $\beta \varphi_x \varphi_y$, while (i) allows us to define the random cluster representation (see Definition \ref{Def: random cluster}), and (iii) ensures that the model is non-degenerate.
We now define the spin model at finite volume. We state the definition for nearest-neighbour interactions, but our methods also apply to the model with finite-range ferromagnetic interactions that are invariant under a group acting transitively on $V$. 
\begin{definition}
    Let $\Lambda$ be a finite subset of $V$, $\beta \geq 0$, $\rho$ an admissible single-site measure, and $\eta \in \mathbb{R}^{V}$.
    The finite-volume spin model on $\Lambda$ at inverse temperature $\beta \geq 0$ with single-site measure $\rho$ and boundary conditions $\eta$ is the measure $\nu_{\Lambda, \beta, \rho}^{\eta}$ on $\mathbb{R}^{\Lambda}$ whose density at a spin configuration $\varphi \in \mathbb{R}^{\Lambda}$ is given by
\begin{equation*}
    \mathrm{d} \nu_{\Lambda, \beta, \rho}^{\eta}[\varphi] = \frac{1}{Z_{\Lambda, \beta, \rho}^{\eta}}\exp(- \beta H_{\Lambda}^{\eta}(\varphi)) \prod_{x \in \Lambda} \mathrm{d} \rho(\varphi_x),
\end{equation*}
where the partition function $Z_{\Lambda, \beta, \rho}^{\eta}$ is the normalising constant that makes $\nu_{\Lambda, \beta, \rho}^{\eta}$ a probability measure, and $H_{\Lambda}^{\eta}(\varphi)$ is the Hamiltonian, given by
\begin{equation*}
    H_{\Lambda}^{\eta}(\varphi) = - \sum_{\substack{xy \in E\\x, y \in \Lambda}} \varphi_x \varphi_y - \sum_{\substack{xy \in E \\ x \in \Lambda, \, y\in V \setminus \Lambda}} \varphi_x \eta_y.
\end{equation*}
\end{definition}
\noindent We will sometimes write $\langle \cdot \rangle_{\Lambda, \beta, \rho}^{\eta}$ for the expectation with respect to the measure $\nu_{\Lambda, \beta, \rho}^{\eta}$, and will frequently drop $\rho$ from the notation.

This framework includes several well-studied models in statistical physics:
\begin{enumerate}
    \item[$(i)$] The Ising and $\varphi^4$ models, by choosing, respectively
    \begin{equation*}
        \rho=\delta_{-1}+\delta_1, \qquad \mathrm{d}\rho(\varphi_x)=\exp(-g\varphi_x^4-a\varphi_x^2)\mathrm{d}\varphi_x,
    \end{equation*}
    where for $t\in \mathbb R$, $\delta_t$ is the Dirac measure at $t$, and where $g>0$ and $a\in \mathbb R$.
   \item[$(ii)$] The Blume--Capel model, by choosing
    \begin{equation*}
    \rho=\exp(\Delta)\delta_{-1}+\delta_0+\exp(\Delta)\delta_1,
    \end{equation*}
    where $\Delta\in \mathbb R$.
    \item[$(iii)$] General $P(\varphi)$ models, by choosing
    \begin{equation*}
    \mathrm{d}\rho(\varphi_x)=\exp(-P(\varphi_x))\mathrm{d}\varphi_x,
    \end{equation*}
    where $P$ is an even polynomial of degree at least $4$ and of positive leading coefficient.
\end{enumerate}

\subsection{Infinite-volume measures and regularity}
\label{section: infinite-volume}

We now introduce infinite-volume Gibbs measures via the DLR equation.
\begin{definition}
    \label{def: Gibbs}
    Let $\beta$ and $\rho$ be fixed.
    An infinite-volume Gibbs measure is a probability measure $\nu$ on $\mathbb{R}^{V}$, with the $\sigma-$algebra generated by Borel events depending on finitely many vertices, such that for any finite $\Lambda \subset V$ and any bounded measurable function $g: \mathbb{R}^{\Lambda} \rightarrow \mathbb{R}$, the DLR equation holds:
    \begin{equation*}
        \nu[g] = \int_{\eta \in \mathbb{R}^{V}} \langle g\rangle _{\Lambda, \beta, \rho}^{\eta} \mathrm{d} \nu(\eta).
    \end{equation*} 
\end{definition}

Understanding the properties and structure of the set of Gibbs measures is a key question in the study of the model. In contrast to the Ising model, a potential difficulty arises in the case of unbounded spins when working with arbitrary Gibbs measures, since it is a priori possible that the measure is supported on configurations that grow very rapidly.
To avoid such pathological measures, we restrict to \emph{regular} infinite-volume Gibbs measures. To define this class of measures, the important observation is that the finite-volume measure $\nu_{\Lambda, \beta, \rho}^{\eta}$ satisfies a regularity property which allows us to bound the Radon-Nikodym derivative of the system in a finite subset $\Lambda \subset V$ with respect to a product measure, up to constants that depend on the boundary conditions $\eta$. Provided that $\eta$ grows at most exponentially in the distance from the origin, this allows us to prove weak convergence of $\nu_{\Lambda, \beta, \rho}^{\eta}$ as $\Lambda \nearrow V$ to an infinite-volume Gibbs measure that satisfies a similar regularity property. Moreover, for fixed $\beta, \rho$ any positive boundary conditions growing sufficiently quickly produce the same measure in the infinite-volume limit, which we call the (infinite-volume) plus measure and denote by $\nu_{\beta, \rho}^{+}$.

We now define some notation needed to state the regularity result. Consider a finite subset $\Lambda \subset V$ and define the interior and exterior boundary $\partial \Lambda = \{x \in \Lambda : d_G(x, V \setminus \Lambda) = 1\}$ and $\partial^{\mathrm{ext}} \Lambda = \{x \in V \setminus \Lambda : d_G(x, \Lambda) = 1\}$.  For $x \in \Lambda, \eta \in \mathbb{R}^{V}$, and $a>0$, let
\begin{equation*}
    A(x, \Lambda, \eta, \beta, a) = \max \left\{1, \max_{y \in \partial^{\mathrm{ext}} \Lambda} |\eta_y| \left(\frac{2D\beta}{a}\right)^{d_G(x,y)} \right\}.
\end{equation*}
The function $A$ can be thought of as representing the influence of the boundary conditions $\eta$ on the value of the spin at $x$, similar to the Cameron-Martin formula for Gaussian fields. Note that if we take $a>2D\beta$, then the term $\left(\frac{2D\beta}{a}\right)^{d_G(x,y)}$ above decays exponentially fast.  
Thus for any $\varepsilon > 0$ and any boundary conditions that grow at most exponentially fast, we can choose $a$ to ensure that we have $A(x, \Lambda_n, \eta, \beta, a) = 1$ for all $n \geq 0$ and all $x \in \Lambda_n$ with $d_G(x, \Lambda_n) > \varepsilon n$.

We first state regularity for the finite-volume spin model, which follows from \cite[Theorem 4.1]{PV26}.
See also \cite{Lebowitz_Presutti, random_tangled} for related regularity results. Below $\rho_a$ is the single-site measure defined by $\mathrm{d} \rho_a(u) = e^{au^2} \mathrm{d} \rho(u)$, which is admissible if $\rho$ is admissible.

\begin{proposition}
    \label{prop: regularity}
    For any $\beta \geq 0$ and $a \geq 2 D \beta$, there exists $B > 0$ such that for any finite $\Lambda \subset V$, $\Lambda' \subset \Lambda$, $\psi \in \mathbb{R}^{\Lambda'}$, and any boundary conditions $\eta \in \mathbb{R}^{V}$,
    \begin{equation}
        \label{eq: regularity statement}
        \mathrm{d} \nu_{\Lambda, \beta, \rho}^{\eta}[\varphi|_{\Lambda'} = \psi] \leq \prod_{x \in \Lambda'} e^{BA(x, \Lambda, \eta, \beta, a)^2} \mathrm{d} \nu_{\Lambda', 0, \rho_{a}}^{0}[\psi].
    \end{equation}
\end{proposition}

Let us remark that the measure $\nu_{\Lambda', 0, \rho_{a}}^{0}$ above is a product measure, since $\beta=0$. Motivated by \eqref{eq: regularity statement}, we define regularity for infinite-volume measures as follows.

\begin{definition}
    We say an infinite-volume measure $\nu$ on $\mathbb{R}^{V}$ is \emph{regular} (with respect to $\rho$) if there exist $a, B > 0$ such that for any finite $\Lambda \subset V$ and any $\psi \in \mathbb{R}^{\Lambda}$,
    \begin{equation}
        \label{eq: infinite-vol regularity}
        \mathrm{d} \nu[\varphi|_{\Lambda} = \psi] \leq e^{B} \mathrm{d} \nu_{\Lambda, 0, \rho_{a}}^{0}[\psi].
    \end{equation}
\end{definition}

We now introduce the plus measure, which can be thought of as the maximal regular infinite-volume Gibbs measure. Recall that $\Lambda_k$ denotes the ball of radius $k$ in graph distance and let $\eta^{+} \in \mathbb{R}^{V}$ be given by $\eta^{+}_{x} = \sqrt{\mathrm{\log}(|\Lambda_{d_{G}(o,x)}|)}$. It was proved in \cite[Proposition 5.8]{PV26} that the measures $\nu_{\Lambda, \beta, \rho}^{\eta^{+}}$ converge weakly to a regular Gibbs measure as $\Lambda \nearrow V$, so we can define the plus measure
\begin{equation*}
    \nu_{\beta, \rho}^{+} \coloneq \lim_{\Lambda \nearrow V} \nu_{\Lambda, \beta, \rho}^{\eta^{+}}.
\end{equation*}
We mention that an alternative way to define $\nu_{\beta, \rho}^{+}$ by taking the limit of finite-volume measures with free boundary conditions and a shifted single-site measure at the boundary of $\Lambda$ was considered in \cite[Section 5.4]{PV26}. An advantage of this alternative construction is that the finite-volume measures are regular up to the boundary.

In the next proposition, we state some useful properties of the plus measure. Below, we say that $\mu$ is \textit{stochastically dominated} by $\nu$, and write $\mu \preceq \nu$, if $\mu[f] \leq \nu[f]$ for any increasing function $f$.
\begin{proposition}
    \label{Prop: plus measure spin model}
    Let $\beta \geq 0$ and let $\rho$ be an admissible single-site measure.
    \begin{enumerate}[(i)]
        \item Any regular infinite-volume Gibbs measure is stochastically dominated by $\nu^{+}_{\beta, \rho}$.
        \item For any boundary conditions $\eta \in \mathbb{R}^{V}$ such that there exists $\lambda \in \mathbb{R}$ with $\eta^{+}_x \leq \eta_x \leq e^{\lambda d_{G}(o, x)}$ for all $x \in V$,
        $\nu_{\Lambda, \beta, \rho}^{\eta} \rightarrow \nu_{\beta, \rho}^{+}$ as $\Lambda \nearrow V$.
    \end{enumerate}
\end{proposition}
\begin{proof}
    See \cite[Proposition 5.8]{PV26} for a proof of (i).
    For (ii), notice that the assumption on $\eta$ implies that for any fixed $x \in V$ and any $a$ large enough, $A(x, \Lambda, \eta, \beta, a) \rightarrow 1$ as $\Lambda \nearrow V$, so by \eqref{eq: regularity statement} the family of measures $(\nu_{\Lambda, \beta, \rho}^{\eta})_{\Lambda \Subset V}$ is tight. Any subsequential limit $\nu$ is regular by \eqref{eq: regularity statement} and is a Gibbs measure by the domain Markov property for the measures $\nu_{\Lambda, \beta, \rho}^{\eta}$, so $\nu \preceq \nu_{\beta, \rho}^{+}$ by (i). On the other hand, $\nu_{\Lambda, \beta, \rho}^{\eta} \succeq \nu_{\Lambda, \beta, \rho}^{\eta^+}$ for all finite $\Lambda \subset V$ by monotonicity in boundary conditions (see Proposition \ref{prop:stoc_monotonicity} below), so $\nu \succeq \nu_{\beta, \rho}^{+}$. We thus obtain $\nu_{\Lambda, \beta, \rho}^{\eta}[H] \rightarrow \nu_{\beta, \rho}^{+}[H]$ for any increasing event $H$ depending on finitely many spins, and full convergence follows because these events generate the $\sigma-$algebra.
\end{proof}

\subsection{Correlation inequalities}\label{sec:correlation_ineq}

We now present various classical correlation inequalities that will be used in this paper (see \cite{random_tangled} and references therein for proofs and background). We begin with correlations of increasing functions.

\begin{proposition}[FKG inequalities, {\cite[Propositions~B.1--.2]{random_tangled} and \cite[Corollary~6.4]{LammersOtt2021}}] \label{prop: spin model FKG} Let $\Lambda \subset V$ be finite, $\beta\geq0$,  and $\eta\in \mathbb R^{V}$. Then, for any increasing and bounded functions $F,G:\mathbb{R}^\Lambda\to \mathbb{R}$,
\begin{equation*}
    \langle F(\varphi)G(\varphi)\rangle^{\eta}_{\Lambda,\beta}\geq \langle F(\varphi)\rangle^{\eta}_{\Lambda,\beta}\langle G(\varphi)\rangle^{\eta}_{\Lambda,\beta}.
\end{equation*}
The same holds for the absolute value field if $\eta_x\geq 0$ for every $x \in \Lambda$, namely
\begin{equation*}
    \langle F(|\varphi|)G(|\varphi|) \rangle^{\eta}_{\Lambda,\beta}\geq \langle F(|\varphi|) \rangle^{\eta}_{\Lambda,\beta} \langle G(|\varphi|) \rangle^{\eta}_{\Lambda,\beta}.
\end{equation*}
\end{proposition}

As a consequence of the FKG inequalities above, we obtain the following monotonicity properties of correlations: the first inequality is standard, and the second inequality follows from a straightforward adaptation of the proof of \cite[Lemma 2.13]{random_tangled}. 

\begin{proposition}\label{prop:stoc_monotonicity}
Let $\Lambda \subset V$ be finite, and $\beta\geq0$. Let also $\eta_1,\eta_2\in \mathbb R^{V}$ be such that $\eta_1\leq \eta_2$. Then, for any increasing function $F:\mathbb{R}^{\Lambda}\to\mathbb{R}$ we have 
\begin{equation*}
\langle F(\varphi) \rangle^{\eta_1}_{\Lambda,\beta}\leq \langle F(\varphi) \rangle^{\eta_2}_{\Lambda,\beta}.  
\end{equation*}
The same holds for the absolute value field if $|\eta_1| \leq \eta_2$, namely for any increasing function $F: (\mathbb{R}^+)^{\Lambda}\to\mathbb{R}$ we have 
\begin{equation*}
\langle F(|\varphi|) \rangle^{\eta_1}_{\Lambda,\beta}\leq \langle F(|\varphi|) \rangle^{\eta_2}_{\Lambda,\beta}. 
\end{equation*}
\end{proposition}

We now turn to spin correlations. In what follows, given $\Lambda\subset V$ finite and $A:\Lambda\rightarrow \mathbb{N}$, we write $\varphi_A\coloneqq \prod_{x\in \Lambda} \varphi_x^{A_x}$.

\begin{proposition}\label{prop:monotonicity in bc} Let $A: \Lambda\rightarrow \mathbb{N}$. Then, for any $\eta_1,\eta_2\in \mathbb R^{V}$ such that $|\eta_1|\leq \eta_2$,
\begin{equation*}
    \langle \varphi_A\rangle^{\eta_1}_{\Lambda,\beta}\leq \langle \varphi_A\rangle^{\eta_2}_{\Lambda,\beta}.
\end{equation*}
\end{proposition}
\begin{proof}
Let $\textup{sgn}(t)=\mathbbm{1}_{t\geq 0}-\mathbbm{1}_{t<0}$ and write $\sigma_x=\textup{sgn}(\varphi_x)$. Conditionally on $|\varphi|$, the sign field $(\sigma_x)_{x\in \Lambda}$ is distributed according to an Ising model. For the Ising model, the statement of the proposition follows from the monotonicity inequality of Lebowitz \cite{Lebowitz1977CoexistencePhasesIsing}, i.e.\ for any $\psi\in (\mathbb{R}^+)^{\Lambda}$
\[
\langle \sigma_A \mid |\varphi|=\psi\rangle^{\eta_1}_{\Lambda,\beta}\leq \langle \sigma_A \mid |\varphi|=\psi\rangle^{\eta_2}_{\Lambda,\beta}.
\]
Moreover, $\langle \sigma_A \mid |\varphi|=\psi\rangle^{\eta_2}_{\Lambda,\beta}$ is an increasing function of $\psi$ by monotonicity in the coupling constants, since $\eta_2\geq 0$. Then we use the monotonicity in the absolute value field of Proposition~\ref{prop:stoc_monotonicity} to conclude that 
\[
\langle \varphi_A\rangle^{\eta_1}_{\Lambda,\beta}\leq \langle |\varphi|_A \langle \sigma_A\mid |\varphi|\rangle^{\eta_2}_{\Lambda,\beta}\rangle^{\eta_1}_{\Lambda,\beta}\leq \langle |\varphi|_A \langle \sigma_A\mid |\varphi|\rangle^{\eta_2}_{\Lambda,\beta}\rangle^{\eta_2}_{\Lambda,\beta}=\langle \varphi_A\rangle^{\eta_2}_{\Lambda,\beta},
\]
as desired.
\end{proof}

\section{Random cluster measure}\label{section: random cluster}

In this section, we introduce the random cluster representation of our spin models, the Edwards--Sokal coupling that relates them, and certain correlation inequalities that will be useful throughout the paper.

\subsection{Definition of the random cluster measure}

Consider a finite subset $\Lambda \subset V$ and let $\overline{E}(\Lambda) = \{xy \in E : x \in \Lambda \text{ or } y \in \Lambda\}$. Recall that $\partial^{\mathrm{ext}} \Lambda = \{x \in V \setminus \Lambda : d_G(x, \Lambda) = 1\}$ and set $\overline{\Lambda} = \Lambda \cup \partial^{\mathrm{ext}}\Lambda$. We define boundary conditions on $\Lambda$ as a pair $(\xi, \sfb)$, where $\xi = \{\xi_1, \ldots, \xi_{|\xi|}\}$ is a partition of $\partial^{\mathrm{ext}} \Lambda$, and $\sfb \in (\mathbb{R}^{+})^{\partial^{\mathrm{ext}} \Lambda}$.
We call the connected components of the random subgraph $(\overline{\Lambda}, \{e \in \overline{E}(\Lambda) : \omega_e = 1\})$ the clusters of $\omega$, where we identify any clusters that contain vertices in the same element of the partition $\xi$. Denote by $k^{\xi}(\omega)$ the number of such clusters. We say $x$ is connected to $y$ and write $x \xleftrightarrow{\omega} y$ (or just $x \longleftrightarrow y$) if $x$ and $y$ are in the same cluster of $\omega$.

A random cluster representation for the $\varphi^{4}$ model on $\mathbb{Z}^{d}$ was introduced in \cite{well_behaved}, and it can be defined in the same way for the spin model on $G$ with any admissible single-site measure.

\begin{definition}
\label{Def: random cluster}
Let $\Lambda \subset V$ be finite with boundary conditions $(\xi, \sfb)$ on $\Lambda$, $\beta \geq 0$, and $\rho$ an admissible single-site measure.
The random cluster measure on $\Lambda$ is the probability measure $\Psi_{\Lambda, \beta, \rho}^{(\xi, \sfb)}$ on pairs $(\sfa, \omega) \in (\mathbb{R}^{+})^{\overline{\Lambda}} \times \{0,1\}^{\overline{E}(\Lambda)}$ satisfying $\sfa_x = \sfb_x$ for all $x \in \partial^{\mathrm{ext}} \Lambda$ defined by
\begin{equation*}
    \mathrm{d} \Psi_{\Lambda, \beta, \rho}^{(\xi, \sfb)}[(\mathsf{a}, \omega)] = \frac{\mathbbm{1}_{\{\sfa|_{\partial^{\mathrm{ext}} \Lambda} = \sfb\}}2^{k^{\xi}(\omega)}}{Z_{\Lambda, \beta, \rho}^{(\xi, \sfb)}} \left( \prod_{xy \in \overline{E}(\Lambda)} \sqrt{1-p_{xy}(\mathsf{a})} \left(\frac{p_{xy}(\mathsf{a})}{1-p_{xy}(\mathsf{a})}\right)^{\omega_{xy}}\right) \prod_{x \in \Lambda} \mathrm{d} \rho(\sfa_x),
\end{equation*}
where $p_{xy}(\mathsf{a}) = 1 - \exp(-2\beta \mathsf{a}_x \mathsf{a}_y)$ and $Z_{\Lambda, \beta, \rho}^{(\xi, \sfb)}$ is the normalising constant that makes $\Psi_{\Lambda, \beta, \rho}^{(\xi, \sfb)}$ a probability measure. We will write $\chi_{\Lambda, \beta, \rho}^{(\xi, \sfb)}$ and $\Phi_{\Lambda, \beta, \rho}^{(\xi, \sfb)}$ for the marginal distributions of $\sfa$ and $\omega$ respectively. As for the spin model we will often omit $\rho$ from the notation. We also allow the boundary field $\sfb$ to be an element of $(\mathbb{R}^{+})^{V}$ instead of $(\mathbb{R}^{+})^{\partial^{\mathrm{ext}} \Lambda}$, as this is more convenient when defining boundary conditions for a family of measures with varying $\Lambda$. 
\end{definition}

We define two special kinds of boundary conditions: the free boundary conditions $(\xi, \sfb)$ with $\xi = \{ \{x\} : x \in \partial^{\mathrm{ext}} \Lambda\}$ and $\sfb \equiv 0$, and the wired boundary conditions with $\xi = \{\partial^{\mathrm{ext}} \Lambda\}$. We write $\Psi_{\Lambda, \beta}^0$ for the random cluster measure with free boundary conditions and $\Psi_{\Lambda, \beta}^{(w, \sfb)}$ for the random cluster measure with wired boundary conditions. 

As observed in \cite{well_behaved}, it follows from Definition \ref{Def: random cluster} that $\Psi^{(\xi, \sfb)}_{\Lambda, \beta}$ satisfies the domain Markov property when conditioning on the value of $\sfa$ outside a set of vertices $\Lambda' \subset \Lambda$ and the value of $\omega$ on edges with both endpoints outside $\Lambda'$,
meaning that
for every $\sfa' \in (\mathbb{R}^{+})^{\Lambda'}, \mathsf{a}'' \in (\mathbb{R}^{+})^{\overline{\Lambda} \setminus\Lambda'}$ and every $\omega' \in \{0, 1\}^{\overline{E}(\Lambda')}, \omega'' \in \{0, 1\}^{\overline{E}(\Lambda) \setminus \overline{E}(\Lambda')}$, we have
\begin{equation}
    \label{eq: DMP}
    \mathrm{d} \Psi^{(\xi, \sfb)}_{\Lambda, \beta} \left[\sfa|_{\Lambda'} = \sfa', \omega|_{\overline{E}(\Lambda')} = \omega' \ \Big| \ \sfa|_{\overline{\Lambda} \setminus \Lambda'} = \sfa'', \omega|_{\overline{E}(\Lambda) \setminus \overline{E}(\Lambda')} = \omega'' \right] = \mathrm{d} \Psi_{\Lambda', \beta}^{(\xi^{\omega''}, \sfa'')}[(\sfa', \omega')],
\end{equation}
where $\xi^{\omega''}$ is the partition of $\partial^{\mathrm{ext}} \Lambda'$ in which two vertices are in the same element if and only if they are in the same cluster of $\omega''$ in the graph $(\overline{\Lambda} \setminus \Lambda', \overline{E}(\Lambda) \setminus \overline{E}(\Lambda'))$ after identifying clusters containing vertices in the same element of $\xi$.

\subsection{Edwards--Sokal coupling}
In this section, we introduce the coupling between the spin model and the random cluster representation. We first describe informally the idea behind the coupling, before stating the formal result in Proposition \ref{prop: coupling} below.

Consider $\varphi \sim \nu_{\Lambda, \beta}^{\eta}$, where the boundary conditions $\eta$ are non-negative, and let $\sfa_x = |\varphi_x|$ and $\sigma_x = \mathrm{sgn}(\varphi_x)$. The assumption that $\rho$ is even means that conditionally on the absolute value field $\sfa$, $\sigma$ is distributed according to an Ising model on $\Lambda$ with random coupling constants $J_{x,y} = \sfa_x \sfa_y$. Using the Edwards--Sokal coupling \cite{Edwards-Sokal} between the Ising model and FK random cluster percolation, we can sample a percolation configuration $\omega$ by setting $\omega_{xy} = 0$ for any edge $xy$ such that $\sigma_x \neq \sigma_y$, and $\omega_{xy} = 1$ with probability $p_{xy}(\sfa)$ independently of all other edges for edges $xy$ with $\sigma_x = \sigma_y$. The random cluster measure $\Psi_{\Lambda, \beta}^{(w, \eta)}$ is defined so that the resulting pair $(\sfa, \omega)$ will be distributed according to $\Psi_{\Lambda, \beta}^{(w, \eta)}$.

Let us now define the Ising and random cluster models on a subgraph $H = (V(H), E(H))$ of $G$ with inhomogeneous coupling constants. Let $\beta \geq 0$, $\sfa \in (\mathbb{R}^{+})^{V(H)}$, and let $\xi$ be a partition of $\partial H \coloneq \{x \in V(H) : xy \in E \text{ for some } y \in V \setminus V(H)\}$. Recall that $p_{xy}(\sfa) = 1-e^{-2 \beta \sfa_x \sfa_y}$ and $k^{\xi}(\omega)$ denotes the number of clusters in $\omega$ after identifying any clusters that contain vertices in the same element of the partition $\xi$. The random cluster model on $H$ with coupling constants given by $\sfa \in (\mathbb{R}^{+})^{V(H)}$ is the probability measure  $\phi^{\xi}_{H, \beta, \sfa}$ on configurations $\omega \in \{0, 1\}^{E(H)}$ given by

\begin{equation*}
    \phi^{\xi}_{H, \beta, \sfa}[\omega] = \frac{1}{Z^{\mathrm{RC}, \xi}_{H, \beta, \sfa}} \prod_{xy \in E(H)} \left( \frac{p_{xy}(\sfa)}{1 - p_{xy}(\sfa)} \right)^{\omega_{xy}} 2^{k^{\xi}(\omega)},
\end{equation*}
where $Z^{\mathrm{RC}, \xi}_{H, \beta, \sfa}$ is a normalising constant.

The Ising model $\langle \cdot \rangle_{H, \beta, \sfa}^{\mathrm{Ising}, \xi}$ on $H$ is the probability measure on configurations $\sigma \in \{0, 1\}^{V(H)}$ given by
\begin{equation*}
    \langle \sigma \rangle_{H, \beta, \sfa}^{\mathrm{Ising}, \xi}
    = \frac{\mathbbm{1}_{\{\sigma \sim_{\mathrm{ext}} \xi\}}}{Z^{\mathrm{Ising}, \xi}_{H, \beta, \sfa}} \prod_{xy \in E(H)} e^{\beta \sfa_x \sfa_y \sigma_x \sigma_y},
\end{equation*}
where $Z^{\mathrm{Ising}, \xi}_{H, \beta, \sfa}$ is a normalising constant and $\{\sigma \sim_{\mathrm{ext}} \xi\}$ is the event that $\sigma_x$ is constant for every vertex $x$ in the same element of $\xi$. If $H = (\overline{\Lambda}, \overline{E}(\Lambda))$, then we abuse notation and write $\Lambda$ instead of $H$ in the subscripts.

It follows from the Edwards--Sokal coupling that we have the following relation between the partition functions:
\begin{equation}
    \label{eq: partition functions}
    Z^{\mathrm{Ising}, \xi}_{H, \beta, \sfa} = Z^{\mathrm{RC}, \xi}_{H, \beta, \sfa} \prod_{xy \in E(H)} e^{- \beta \sfa_x \sfa_y}.
\end{equation}
Using this relation, we obtain the following coupling between the spin model and its random cluster representation.
\begin{proposition}[{\cite[Proposition 4.6]{well_behaved}}]
    \label{prop: coupling}
    Let $\Lambda \subset V$ finite, $\beta \geq 0$, and $(\xi, \sfb)$ boundary conditions on $\Lambda$. Then
    \begin{equation*}
        \mathrm{d} \Psi_{\Lambda, \beta}^{(\xi, \sfb)}[(\sfa, \omega)] = \phi_{\Lambda, \beta, \sfa}^{\xi}[\omega] \mathrm{d} \chi_{\Lambda, \beta}^{(\xi, \sfb)}[\sfa],
    \end{equation*}
    and
    \begin{equation*}
        \mathrm{d} \chi_{\Lambda, \beta}^{(\xi, \sfb)}[\sfa] = 
        \mathbbm{1}_{\{\sfa|_{\partial^{\mathrm{ext}} \Lambda} = \sfb\}}
        \frac{Z^{\mathrm{Ising}, \xi}_{\Lambda, \beta, \sfa}}{Z_{\Lambda, \beta}^{(\xi, \sfb)}} \prod_{x \in \Lambda} \mathrm{d} \rho(\sfa_x).
    \end{equation*}
    Moreover, $\chi_{\Lambda, \beta}^{(w, \sfb)}$ is the law of the absolute value field under $\nu_{\Lambda, \beta}^{\sfb}$.
\end{proposition}

The Edwards--Sokal coupling allows us to write the correlation between $\varphi_x$ and $\varphi_y$ in terms of the probability that $x$ and $y$ are in the same cluster in the random cluster representation, similarly to the case of the Ising and Potts models (see \cite{Grimmett_RC}). We state this property for the measure with wired boundary conditions because that is where we use it later, but the following also holds for the measure with free boundary conditions.

\begin{corollary}
    \label{Cor: coupling}
    Let $\Lambda \subset V$ finite, $\beta \geq 0$, and $\sfb \in (\mathbb{R}^{+})^{\partial^{\mathrm{ext}} \Lambda}$. For every $x,y \in \Lambda,$ we have
    \begin{align*}
        \nu_{\Lambda, \beta}^{\mathsf{b}} [\varphi_x] = \Psi_{\Lambda, \beta}^{(w, \mathsf{b})}[\sfa_x \mathbbm{1}_{x \longleftrightarrow \partial^{\mathrm{ext}}\Lambda}]
    \end{align*}
    and
    \begin{align*}
    \nu_{\Lambda, \beta}^{\mathsf{b}} [\varphi_x \varphi_y] = \Psi_{\Lambda, \beta}^{(w, \mathsf{b})}[\sfa_x \sfa_y \mathbbm{1}_{x \longleftrightarrow y}].
    \end{align*}
\end{corollary}

\subsection{Infinite-volume measures}
We now define the random cluster measure at infinite volume. Let us first state the definition of an increasing function in the context of the random cluster measure.

\begin{definition}
    \label{Def: increasing}
    We say that $f: (\mathbb{R}^{+})^{\overline{\Lambda}} \times \{0, 1\}^{\overline{E}(\Lambda)} \rightarrow \mathbb{R}$ is an increasing function if for any $\mathsf{a}, \mathsf{a}' \in \mathbb{R}^{\overline{\Lambda}}$ and any $\omega, \omega' \in \{0, 1 \}^{\overline{E}(\Lambda)}$ with $\mathsf{a} \leq \mathsf{a}'$ and $\omega \leq \omega'$ pointwise,
    $f(\mathsf{a}, \omega) \leq f(\mathsf{a}', \omega').$
    We say an event $A$ is an increasing event if $\mathbbm{1}_{A}$ is an increasing function. For measures $\mu, \mu'$ on $(\mathbb{R}^{+})^{\overline{\Lambda}} \times \{0, 1\}^{\overline{E}(\Lambda)}$, we say that $\mu$ is stochastically dominated by $\mu'$ and write $\mu \preceq \mu'$ if $\mu[f] \leq \mu'[f]$ for any bounded increasing function $f$.
\end{definition}

To prove convergence to an infinite-volume measure, we will need to use monotonicity in boundary conditions of the measures $\Psi_{\Lambda, \beta}^{(\xi, \sfb)}$, which was established in \cite[Proposition 4.10]{well_behaved} and which we state in the following proposition. 
\begin{proposition}
    \label{prop: bc monotonicity for RC}
    Let $\Lambda \subset V$ be finite and let $(\xi_1, \sfb_1), (\xi_2, \sfb_2)$ be boundary conditions on $\Lambda$ such that $\sfb_1 \leq \sfb_2$ and each element of the partition $\xi_1$ is a subset of an element in $\xi_2$. Then
    $\Psi_{\Lambda, \beta}^{(\xi_1, \sfb_1)} \preceq \Psi_{\Lambda, \beta}^{(\xi_2, \sfb_2)}$.
\end{proposition}

Recall from Section \ref{section: infinite-volume} that $\eta^{+}_{x} = \sqrt{\log(|\Lambda_{d_G(o, x)}|)}$. We now show that the measures $\Psi_{\Lambda, \beta}^{(w, \eta^{+})}$ converge weakly as $\Lambda \nearrow V$ to a measure that we call $\Psi_{\beta}^{1}$. Moreover, the wired measure with any boundary field $\sfb$ that grows faster than $\eta^{+}$ but not faster than exponentially will also converge to $\Psi_{\beta}^{1}$.

\begin{proposition}\label{prop: plus convergence}
    For any $\beta \geq 0$, there exists a probability measure $\Psi_{\beta}^{1}$ on $(\mathbb{R}^+)^{V} \times \{0, 1\}^{E}$ such that for any $\lambda \in \mathbb{R}$ and any $\sfb \in \mathbb{R}^{V}$ with $\eta^{+}_x \leq \sfb_x \leq e^{\lambda d_G(o, x)}$ for all $x \in V$,
    \begin{equation*}
        \Psi_{\Lambda, \beta}^{(w, \sfb)} \rightarrow \Psi_{\beta}^{1}
        \text{ as } \Lambda \nearrow V.
    \end{equation*}
\end{proposition}

\begin{proof}
    Let $\Lambda' \subset \Lambda$ be finite subsets of $V$ and let $H$ be an increasing event depending only on $(\sfa|_{\Lambda'}, \omega|_{\overline{E}(\Lambda')})$. Let $F_{\Lambda'}$ be the event that $\sfa_x \leq \eta^{+}_{x}$ for all $x$ in $\partial^{\mathrm{ext}}\Lambda'$.
    Applying the domain Markov property and Proposition \ref{prop: bc monotonicity for RC}, we get
    \begin{equation}
        \label{eq: plus convergence 0}
        \Psi_{\Lambda, \beta}^{(w, \sfb)}[H] = \int_{(\mathbb{R}^+)^{\overline{\Lambda}} \times \{0,1\}^{\overline{E}(\Lambda)}} \Psi_{\Lambda', \beta}^{(w(\Lambda)^{\tilde{\omega}}, \tilde{\sfa})}[H] \mathrm{d} \Psi_{\Lambda, \beta}^{(w, \sfb)}[(\tilde{\sfa}, \tilde{\omega})]
        \leq \Psi_{\Lambda', \beta}^{(w, \eta^{+})}[H] + \chi_{\Lambda, \beta}^{(w, \sfb)}[F_{\Lambda'}^{\mathsf{c}}],
    \end{equation}
    where $w(\Lambda)^{\tilde{\omega}}$ is the partition of $\partial^{\mathrm{ext}} \Lambda'$ where two vertices are in the same element if and only if they are in the same cluster of $\tilde{\omega}$ in the graph $(\overline{\Lambda} \setminus \Lambda', \overline{E}(\Lambda) \setminus \overline{E}(\Lambda'))$ after identifying all clusters that intersect $\partial \Lambda$.
    If we can show that 
    \begin{equation}
        \label{eq: plus convergence 1}
        \lim_{\Lambda' \nearrow V} \lim_{\Lambda \nearrow V} \chi_{\Lambda, \beta}^{(w, \sfb)}[F_{\Lambda'}^{\mathsf{c}}] = 0,
    \end{equation}
    then taking first $\limsup_{\Lambda \nearrow V}$ and then $\liminf_{\Lambda' \nearrow V}$ in \eqref{eq: plus convergence 0} and using the monotonicity in boundary conditions of Proposition \ref{prop: bc monotonicity for RC}, we have that $\Psi_{\Lambda, \beta}^{(w, \sfb)}[H]$ and $\Psi_{\Lambda, \beta}^{(w, \eta^{+})}[H]$ converge to the same limit, which we define to be $\Psi_{\beta}^{1}[H]$. We can then extend $\Psi_{\beta}^{1}$ to a measure on the $\sigma-$algebra generated by increasing events depending on finitely many vertices and edges. 
    
    To prove \eqref{eq: plus convergence 1}, note that by the Edwards--Sokal coupling, $\chi_{\Lambda, \beta}^{(w, \sfb)}$ is the law of the absolute value field in the spin model $\nu_{\Lambda, \beta}^{\sfb}$, so by Proposition \ref{Prop: plus measure spin model},
    \begin{equation*}
        \lim_{\Lambda \nearrow V} \chi_{\Lambda, \beta}^{(w, \sfb)} [F_{\Lambda'}^{\mathsf{c}}] 
        = \nu^{+}_{\beta}[|\varphi_x| \geq \eta^{+}_x \text{ for some } x \in \partial^{\mathrm{ext}} \Lambda'].
    \end{equation*}
    The regularity property \eqref{eq: infinite-vol regularity} of $\nu_{\beta}^{+}$ together with a union bound yields for some constants $a, B >0$,
    \begin{align*}
        \lim_{\Lambda \nearrow V} \chi_{\Lambda, \beta}^{(w, \sfb)} [F_{\Lambda'}^{\mathsf{c}}]
         &\leq e^{B} \sum_{x \in \partial^{\mathrm{ext}} \Lambda'} \nu_{\{x\}, 0, \rho_a}^{0} [|\varphi_x| \geq \eta_{x}^+].
    \end{align*}
    Now let $g(\Lambda') = \min_{x \in \partial^{\mathrm{ext}} \Lambda'} d_G(o, x)$. Applying Markov's inequality to the random variable $e^{3 |\varphi_x|^2}$, we obtain
    \begin{align*}
        \lim_{\Lambda \nearrow V} \chi_{\Lambda, \beta}^{(w, \sfb)} [F_{\Lambda'}^{\mathsf{c}}]
        &\leq e^{B} \nu_{\{x\}, 0, \rho_a}^{0}[e^{3|\varphi_x|^2}] \sum_{x \in \partial^{\mathrm{ext}} \Lambda'} e^{-3 |\eta_{x}^+|^2}\\
        &\leq e^{B} \nu_{\{x\}, 0, \rho_a}^{0}[e^{3|\varphi_x|^2}] \sum_{k = g(\Lambda')}^{\infty} |\Lambda_k|^{-2} \longrightarrow 0 \text{ as } \Lambda' \nearrow V.
    \end{align*}
\end{proof}

\begin{remark}
    \label{remark: infinite coupling}
    The coupling between the spin model and its random cluster representation extends to the infinite-volume measures $\nu_{\beta}^{+}$ and $\Psi_{\beta}^{1}$. In particular, from Corollary \ref{Cor: coupling} we get $ \nu_{\beta}^{+} [\varphi_x] = \Psi_{\beta}^{1}[\sfa_x \mathbbm{1}_{x \longleftrightarrow \infty}]$, where $x \longleftrightarrow \infty$ is the event that the cluster containing $x$ is infinite, so the critical point $\beta_c =\inf\{\beta\geq 0: \, \nu_{\beta}^{+} [\varphi_o] > 0\}$ for the spin model coincides with the critical point $\inf\{\beta \geq 0 : \, \Psi_{\beta}^{1}[o \longleftrightarrow \infty] > 0\}$ for the existence of an infinite cluster in the random cluster representation.
\end{remark}

\subsection{Correlation inequalities}
\label{Section: monotonicity}
Unlike the classical random cluster model, the edge marginal $\Phi_{\Lambda, \beta}^{(\xi, \sfb)}$ of our model does not in general satisfy monotonicity when conditioning on the state of edges. The main goal of this subsection is to show that $\Psi_{\Lambda, \beta}^{(\xi, \sfb)}$ does satisfy a weaker monotonicity property when conditioning on the states of edges together with their endpoints, encapsulated in Definition \ref{Def: weak monotonicity} below. Along the way we will prove some useful correlation inequalities of the random cluster representation.
\begin{definition}
    \label{Def: weak monotonicity}
    Let $\mu$ be a measure on pairs $(\sfa, \omega) \in (\mathbb{R}^{+})^{\overline{\Lambda}} \times \{0,1\}^{\overline{E}(\Lambda)}$ and recall Definition \ref{Def: increasing}.
    \begin{itemize}
    \item We say that $\mu$ is \textit{weakly monotonic} if $\sfa|_{\partial^{\mathrm{ext}}\Lambda}$ is deterministic and for any $\Lambda' \subset \Lambda$, $E' \subset \{xy \in \overline{E}(\Lambda') : x,y \in \Lambda' \cup \partial^{\mathrm{ext}} \Lambda\}$, and any $\sfa_1, \sfa_2 \in (\mathbb{R}^{+})^{\Lambda'}, \omega_1, \omega_2 \in \{0, 1\}^{E'}$ with $\sfa_1 \leq \sfa_2$ and $\omega_1 \leq \omega_2$ pointwise,
    \begin{equation*}
        \mu[\ \cdot \mid \sfa|_{\Lambda'} = \sfa_1, \omega|_{E'} = \omega_1] \preceq \mu[\ \cdot \mid \sfa|_{\Lambda'} = \sfa_2, \omega|_{E'} = \omega_2].
    \end{equation*}
    \item
    We say that $\mu$ satisfies the FKG inequality if for any increasing square-integrable functions $f, g: (\mathbb{R}^{+})^{\overline{\Lambda}} \times \{0, 1\}^{\overline{E}(\Lambda)} \rightarrow \mathbb{R}$, we have
    $\mu[f \cdot g] \geq \mu[f] \mu[g].$
    \end{itemize}
\end{definition}

In the next proposition, we show that weak monotonicity implies a stronger form of the FKG inequality that allows for conditioning on the states of vertices and on edges whose endpoints are both conditioned on.

\begin{proposition}
    \label{prop: monotonicity implies FKG}
    Let $\mu$ be a weakly monotonic measure on $(\mathbb{R}^{+})^{\overline{\Lambda}} \times \{0,1\}^{\overline{E}(\Lambda)}$. For any $\Lambda' \subset \Lambda$, $E' \subset \{xy \in \overline{E}(\Lambda') : x,y \in \Lambda' \cup \partial^{\mathrm{ext}} \Lambda\}$, and any configurations $\eta \in (\mathbb{R}^{+})^{\Lambda'}$, $\theta \in \{0,1\}^{E'}$, the conditional measure $\mu[\ \cdot \mid \sfa|_{\Lambda'} = \eta, \omega|_{E'} = \theta]$ satisfies the FKG inequality.
\end{proposition}

\begin{proof}
    We proceed by induction on $\Lambda'$ and $E'$, following a similar approach to the proof of \cite[Proposition 4.8]{well_behaved}. First observe that the statement is trivially true when $\Lambda' = \Lambda$ and $E' = \overline{E}(\Lambda)$.
    Now assume that for some $\Lambda' \subset \Lambda$ and $E' \subset \{xy \in \overline{E}(\Lambda') : x,y \in \Lambda' \cup \partial^{\mathrm{ext}} \Lambda\}$, we have for any configurations $\eta \in (\mathbb{R}^+)^{\Lambda'}$, $\theta \in \{0, 1\}^{E'}$ and any increasing square-integrable functions $f, g,$
    \begin{equation*}
        \mu[fg \mid \sfa|_{\Lambda'} = \eta, \omega|_{E'} = \theta] \geq 
        \mu[f \mid \sfa|_{\Lambda'} = \eta, \omega|_{E'} = \theta]
        \mu[g \mid \sfa|_{\Lambda'} = \eta, \omega|_{E'} = \theta].
    \end{equation*}
    
    We aim to show that the FKG inequality is satisfied by the measure conditioned on  the configuration in $(\Lambda', E'\setminus \{xy\})$ for any edge $xy \in E'$, and also by the measure conditioned on the configuration in $(\Lambda' \setminus \{x\}, E')$ for any vertex $x \in \Lambda'$ that is not incident to any edge in $E'$. 
    Since any $(\Lambda', E')$ with $\Lambda' \subset \Lambda$ and $E' \subset \{xy \in \overline{E}(\Lambda') : x,y \in \Lambda' \cup \partial^{\mathrm{ext}} \Lambda\}$ can be reached from $(\Lambda, \overline{E}(\Lambda))$ by repeatedly removing either an edge or a vertex that is not incident to any edge, the desired result follows.
    
    Let us first consider the case when we remove an edge $xy$ from the conditioning. Write $\mu_{\Lambda', E'}^{\eta, \theta} = \mu[\ \cdot \mid \sfa|_{\Lambda'} = \eta, \omega|_{E'} = \theta]$.
    For any $\eta \in (\mathbb{R}^+)^{\Lambda'}$, $\theta \in \{0, 1\}^{E' \setminus \{xy\}}$, sampling two independent configurations $(\sfa, \omega), (\sfa', \omega') \sim \mu_{\Lambda', E' \setminus \{xy\}}^{\eta, \theta}$, we have 
    \begin{equation}
         \label{eq: FKG induction 0}
         \begin{split}
         \mu_{\Lambda', E' \setminus \{xy\}}^{\eta, \theta} \otimes \mu_{\Lambda', E' \setminus \{xy\}}^{\eta, \theta} [(f(\sfa, \omega) - f(\sfa', \omega'))(g(\sfa, \omega) - g(\sfa', \omega'))] = \\
         2 (\mu_{\Lambda', E' \setminus \{xy\}}^{\eta, \theta}[f\cdot g] - \mu_{\Lambda', E' \setminus \{xy\}}^{\eta, \theta}[f] \mu_{\eta, \theta}[g]).
         \end{split}
    \end{equation}
    We claim that for any $s, s' \in \{0,1\}$ and any increasing functions $f, g$,
    \begin{equation}
        \label{eq: FKG induction 1}
        \mu_{\Lambda', E'}^{\eta, \theta \Delta s} \otimes \mu_{\Lambda', E'}^{\eta, \theta \Delta s'} [(f(\sfa, \omega) - f(\sfa', \omega'))(g(\sfa, \omega) - g(\sfa', \omega'))] \geq 0,
    \end{equation}
    where $\theta \Delta t \in \{0, 1\}^{E'}$ is given by $\theta \Delta t|_{E' \setminus \{xy\}} = \theta|_{E' \setminus \{xy\}}$ and $(\theta \Delta t)_{xy} = t$. 
    After averaging \eqref{eq: FKG induction 1} over $s$ and $s'$, we obtain that the left hand side of \eqref{eq: FKG induction 0} is positive, which implies that $\mu_{\Lambda', E' \setminus \{xy\}}^{\eta, \theta}$ satisfies the FKG inequality. Since $\eta$ and $\theta$ were arbitrary, this gives the desired result for $(\Lambda', E' \setminus \{xy\}$). To prove \eqref{eq: FKG induction 1}, let $\tilde{\mu}^{t} = \mu_{\Lambda', E'}^{\eta, \theta \Delta t}$ and observe that
    \begin{align}
    \nonumber
    \text{LHS of }\eqref{eq: FKG induction 1}
    &=	
    \tilde{\mu}^{s}[fg] + \tilde{\mu}^{s'}[fg]- \tilde{\mu}^{s}[f]\tilde{\mu}^{s'}[g]
    - \tilde{\mu}^{s'}[f]\tilde{\mu}^{s}[g]
    \\
    \nonumber
    &=
    \tilde{\mu}^{s}[fg] - \tilde{\mu}^{s}[f]\tilde{\mu}^{s}[g]+\tilde{\mu}^{s'}[fg] - \tilde{\mu}^{s'}[f]\tilde{\mu}^{s'}[g]
    + \Big(\tilde{\mu}^{s}[f]-\tilde{\mu}^{s'}[f]\Big)\Big(\tilde{\mu}^{s}[g]-\tilde{\mu}^{s'}[g]\Big)
    \\
    \label{eq: FKG induction 2}
    &\geq
    \Big(\tilde{\mu}^{s}[f]-\tilde{\mu}^{s'}[f]\Big)\Big(\tilde{\mu}^{s}[g]-\tilde{\mu}^{s'}[g]\Big),
    \end{align}
    where we used the inductive hypothesis in the last inequality. By weak monotonicity, $\tilde{\mu}^{t}[F]$ is an increasing function of $t$ for any increasing function $F$, so $\tilde{\mu}^{s}[f]-\tilde{\mu}^{s'}[f]$ and $\tilde{\mu}^{s}[g]-\tilde{\mu}^{s'}[g]$ have the same sign and the last line of \eqref{eq: FKG induction 2} is positive. 

    Now consider the case when we remove a vertex $x$ from the conditioning.
    For $t \geq 0$, let $\eta \Delta t \in (\mathbb{R}^+)^{\Lambda'}$ be given by $\eta \Delta t|_{\Lambda' \setminus \{x\}} = \eta|_{\Lambda' \setminus \{x\}}$ and $(\eta \Delta t)_{x} = t$.
    By the same calculations as above, to show the desired statement for $(\Lambda' \setminus\{x\}, E')$ assuming the case $(\Lambda', E')$ is true, it suffices to check that $\mu_{\Lambda', E'}^{\eta \Delta t, \theta}[F]$ is an increasing function of $t$ for any $\eta \in (\mathbb{R}^{+})^{\Lambda' \setminus \{x\}}$, $\theta \in \{0, 1\}^{E'}$, and any increasing $F$, which is true by weak monotonicity.
\end{proof}

\begin{remark}
    \label{remark: extra vertex removal}
    In the proof of Proposition \ref{prop: monotonicity implies FKG}, the reason we restrict to removing vertices $z$ from $\Lambda'$ that are not incident to any edges in $E'$ is so that we can apply weak monotonicity again to the configuration in $(\Lambda' \setminus \{z\}, E')$. There is no need to make this restriction if we are not going to remove any further vertices or edges from the conditioning, so we can remove one extra vertex from the final conditioning and get that the FKG inequality holds for $\mu[\ \cdot \mid \sfa|_{\Lambda' \setminus \{z\}} = \eta, \omega|_{E'} = \theta]$ for any $\Lambda' \subset \Lambda$, $z\in \Lambda'$, $E' \subset \{xy \in \overline{E}(\Lambda') : x,y \in \Lambda' \cup \partial^{\mathrm{ext}} \Lambda\}$, and any configurations $\eta, \theta$.
\end{remark}

In order to show that $\Psi_{\Lambda, \beta}^{(\xi, \sfb)}$ is weakly monotonic, we first prove that the necessary FKG inequality is satisfied.

\begin{proposition}
    \label{prop: FKG}
    Let $\Lambda \subset V$ be finite with boundary conditions $(\xi, \sfb)$ on $\Lambda$. For any $\Lambda' \subset \Lambda$, $E' \subset \{xy \in \overline{E}(\Lambda') : x,y \in \Lambda' \cup \partial^{\mathrm{ext}} \Lambda\}$, and any configurations $\eta \in (\mathbb{R}^{+})^{\Lambda'}$, $\theta \in \{0,1\}^{E'}$, the conditional measure $\Psi_{\Lambda, \beta}^{(\xi, \sfb)}[\ \cdot \mid \sfa|_{\Lambda'} = \eta, \omega|_{E'} = \theta]$ satisfies the FKG inequality.
\end{proposition}
\begin{proof}
    Write $\mu_{\Lambda', E'}^{\eta, \theta} = \Psi_{\Lambda, \beta}^{(\xi, \sfb)}[\ \cdot \mid \sfa|_{\Lambda'} = \eta, \omega|_{E'} = \theta]$.
    We first consider the case when $f,g$ are increasing functions that depend on $\sfa$ only, and aim to prove that for all $\Lambda' \subset \Lambda$ and $E' \subset \{xy \in \overline{E}(\Lambda') : x,y \in \Lambda' \cup \partial^{\mathrm{ext}} \Lambda\}$,
    \begin{equation}
        \label{eq: FKG induction 3}
        \forall \eta \in (\mathbb{R}^+)^{\Lambda'}, \theta \in \{0,1\}^{E'}, \qquad
        \mu_{\Lambda', E'}^{\eta, \theta}[f(\sfa) g(\sfa)] \geq \mu_{\Lambda', E'}^{\eta, \theta}[f(\sfa)] \mu_{\Lambda', E'}^{\eta, \theta}[g(\sfa)].
    \end{equation}
     
    We proceed by induction on $(\Lambda', E')$, as in the proof of Proposition \ref{prop: monotonicity implies FKG}. By \eqref{eq: FKG induction 0}, \eqref{eq: FKG induction 1}, \eqref{eq: FKG induction 2} (and identical calculations for the measure with a vertex removed from the conditioning) it suffices to show that for any $(\Lambda', E')$ such that \eqref{eq: FKG induction 3} holds, the following two properties are satisfied for any configurations $\eta \in (\mathbb{R}^+)^{\Lambda'}, \theta \in \{0,1\}^{E'}$:
    \begin{enumerate}[(i)]
        \item For any increasing $f: (\mathbb{R}^{+})^{\overline{\Lambda}} \rightarrow \mathbb{R}$, any edge $xy \in E'$, and any $s, s' \in \{0,1\}$ with $s \geq s',$
        \begin{equation*}
            \mu_{\Lambda', E'}^{\eta, \theta \Delta s}[f(\sfa)] \geq \mu_{\Lambda', E'}^{\eta, \theta \Delta s'}[f(\sfa)],
        \end{equation*}
        where $\theta \Delta s$ and $\theta \Delta s'$ are the configurations obtained from $\theta$ by setting $\theta_{xy} = s$ and $\theta_{xy} = s'$ respectively.
        \item
        For any increasing $f: (\mathbb{R}^{+})^{\overline{\Lambda}} \rightarrow \mathbb{R}$, any vertex $x \in \Lambda'$ that is not incident to any edge in $E'$, and any $s, s' \in \mathbb{R}^+$ with $s \geq s',$
        \begin{equation*}
            \mu_{\Lambda', E'}^{\eta \Delta s, \theta}[f(\sfa)] \geq \mu_{\Lambda', E'}^{\eta \Delta s', \theta}[f(\sfa)],
        \end{equation*}
        where $\eta \Delta s$ and $\eta \Delta s'$ are the configurations obtained from $\eta$ by setting $\eta_{x} = s$ and $\eta_{x} = s'$, respectively.
    \end{enumerate}
    Let us first prove (i). Fix $\Lambda', E', \eta, \theta, xy, s, s'$ as above and
    let $G' = (V(G'), E(G'))$ be the graph with edge set $\overline{E}(\Lambda) \setminus E'$ and vertex set consisting of the endpoints of these edges.
    Define the random variable
    \begin{equation*}
        Z(\sfa) = \frac{Z_{G', \beta, \sfa}^{\mathrm{Ising}, \xi^{\theta \Delta s}}}{Z_{G', \beta, \sfa}^{\mathrm{Ising}, \xi^{\theta \Delta s'}}},
    \end{equation*}
    where $\xi^{\theta'}$ is the partition of $\partial G'$ in which two vertices are in the same element if and only if they are in the same cluster of $\theta'$ after identifying clusters containing vertices in the same element of $\xi$.
    Taking the logarithmic derivative of $Z$ with respect to $\sfa_z$ for $z \in \Lambda$, we have
    \begin{equation}
        \label{eq: logarithmic derivative}
        \frac{\mathrm{d}}{\mathrm{d} \sfa_z} \log(Z(\sfa)) = \beta \sum_{\substack{w \in V(G') \\ w \sim z}} \sfa_w \left(  \langle \sigma_w \sigma_z \rangle_{G', \beta, \sfa}^{\mathrm{Ising}, \xi^{\theta \Delta s}} - \langle \sigma_w \sigma_z \rangle_{G', \beta, \sfa}^{\mathrm{Ising}, \xi^{\theta \Delta s'}} \right).
    \end{equation}
    Using the Edwards--Sokal coupling and monotonicity in boundary conditions for the standard random cluster model (\cite[Lemma 4.14]{Grimmett_RC}), \eqref{eq: logarithmic derivative} is positive and hence $Z$ is an increasing function.
    Now by a direct calculation using \eqref{eq: partition functions} we have
    \begin{equation*}
        \mu_{\Lambda', E'}^{\eta, \theta \Delta s}[f(\sfa)]
        = \frac{\mu_{\Lambda', E'}^{\eta, \theta \Delta s'}[f(\sfa) Z(\sfa)]}{\mu_{\Lambda', E'}^{\eta, \theta \Delta s'}[Z(\sfa)]},
    \end{equation*}
    so (i) follows from \eqref{eq: FKG induction 3} applied to $\mu_{\Lambda', E'}^{\eta, \theta \Delta s}$. 

    We now prove (ii). For $\sfa \in (\mathbb{R}^+)^{\overline{\Lambda}}$, let $\sfa \Delta s \in (\mathbb{R}^+)^{\overline{\Lambda}}$ be given by $(\sfa \Delta s)_{x} = s$ and $(\sfa \Delta s)|_{\overline{\Lambda} \setminus \{x\}} = \sfa|_{\overline{\Lambda} \setminus \{x\}}$. A similar calculation using \eqref{eq: partition functions} gives
    \begin{equation*}
        \mu_{\Lambda', E'}^{\eta \Delta s, \theta}[f(\sfa)]
        = \frac{\mu_{\Lambda', E'}^{\eta \Delta s', \theta}[f(\sfa) Z(\sfa)]}{\mu_{\Lambda', E'}^{\eta \Delta s', \theta}[Z(\sfa)]},
        \qquad\text{where }
        Z(\sfa) = \frac{Z_{G', \beta, \sfa \Delta s}^{\mathrm{Ising}, \xi^{\theta}}}{Z_{G', \beta, \sfa}^{\mathrm{Ising}, \xi^{\theta}}}.
    \end{equation*}
    It remains to prove that $Z$ is an increasing function, since (ii) then follows by applying \eqref{eq: FKG induction 3} to $\mu_{\Lambda', E'}^{\eta \Delta s, \theta}$.
    For $z \in \Lambda,$ we have
    \begin{equation}
        \label{eq: logarithmic derivative 2}
        \frac{\mathrm{d}}{\mathrm{d} \sfa_z} \log(Z(\sfa)) = \beta \sum_{\substack{w \in V(G') \\ w \sim z}} (\sfa \Delta s)_w  \langle \sigma_w \sigma_z \rangle_{G', \beta, \sfa \Delta s}^{\mathrm{Ising}, \xi^{\theta}} - \sfa_w\langle \sigma_w \sigma_z \rangle_{G', \beta, \sfa}^{\mathrm{Ising}, \xi^{\theta}}.
    \end{equation}
    By monotonicity of $\langle \sigma_w \sigma_z \rangle_{G', \beta, \sfa}^{\mathrm{Ising}, \xi^{\theta}}$ in the coupling constants $\sfa_x \sfa_y$, we conclude that the above is positive when $\sfa_x = s'$. Hence $Z$ is an increasing function, completing the proof of (ii) and of the statement for the case when $f, g$ depend only on $\sfa$.

    We now consider the case when $f, g : (\mathbb{R}^{+})^{\overline{\Lambda}} \times \{0,1\}^{\overline{E}(\Lambda)} \rightarrow \mathbb{R}$ are increasing functions depending on $\sfa$ and $\omega$. Let $f_{\sfa}(\omega) = f(\sfa, \omega)$ and $g_{\sfa}(\omega) = g(\sfa, \omega)$. Using the FKG inequality for the standard random cluster model (\cite[Theorem 3.8]{Grimmett_RC}),
    \begin{align*}
        \mu_{\Lambda', E'}^{\eta, \theta}[f \cdot g] = \int \phi_{G', \beta, \sfa}^{\xi}[f_{\sfa}(\omega) g_{\sfa}(\omega)] \mathrm{d} \mu_{\Lambda', E'}^{\eta, \theta}[\sfa]
        &\geq \int \phi_{G', \beta, \sfa}^{\xi}[f_{\sfa}(\omega)] \phi_{G', \beta, \sfa}^{\xi}[g_{\sfa}(\omega)]\mathrm{d} \mu_{\Lambda', E'}^{\eta, \theta}[\sfa].
    \end{align*}
    Since $\phi_{G', \beta, \sfa}^{\xi}[f_{\sfa}(\omega)]$ and $\phi_{G', \beta, \sfa}^{\xi}[g_{\sfa}(\omega)]$ are increasing functions of $\sfa$, we can apply \eqref{eq: FKG induction 3} to deduce that the right hand side above is at least $\mu_{\Lambda', E'}^{\eta, \theta}[f] \mu_{\Lambda', E'}^{\eta, \theta}[g]$.
\end{proof}

We now prove that $\Psi_{\Lambda, \beta}^{(\xi, \sfb)}$ is weakly monotonic using the FKG inequality of Proposition \ref{prop: FKG} above.

\begin{proposition}
\label{prop: weak monotonicity}
Let $\Lambda \subset V$ be finite with boundary conditions $(\xi, \sfb)$ on $\Lambda$. Then $\Psi_{\Lambda, \beta}^{(\xi, \sfb)}$ is weakly monotonic.
\end{proposition}

\begin{proof}
    Let $\Lambda', E', \sfa_1, \sfa_2, \omega_1, \omega_2$ be as in Definition \ref{Def: weak monotonicity} and let $G' = (V(G'), E(G'))$ be the graph with edge set $\overline{E}(\Lambda) \setminus E'$ and vertex set consisting of the endpoints of these edges. For $i \in \{1,2\}$, let $\xi_i = \xi^{\omega_i}$ be the partition of $\partial G'$ where two vertices are in the same element if and only if they are in the same cluster of $\omega_i$ after identifying clusters that contain vertices in the same element of $\xi$.
    
    We first prove stochastic domination of the absolute value field. Let $f: (\mathbb{R}^{+})^{\overline{\Lambda}} \rightarrow \mathbb{R}$ be a bounded increasing function that depends only on $\sfa$ and define 
    \[
    Z(\sfa) = \frac{Z_{G', \beta, \sfa \cup \sfa_2}^{\mathrm{Ising}, \xi_2}}{Z_{G', \beta, \sfa \cup \sfa_1}^{\mathrm{Ising}, \xi_1}},
    \]
    where $\sfa \cup \sfa_i$ is equal to $\sfa_i$ on $\Lambda'$ and equal to $\sfa$ elsewhere. 
    Using \eqref{eq: partition functions}, we have
    \begin{equation*}
        \Psi_{\Lambda, \beta}^{(\xi, \sfb)}[f(\sfa) \mid \sfa|_{\Lambda'} = \sfa_2, \omega|_{E'} = \omega_2] = \frac{\Psi_{\Lambda, \beta}^{(\xi, \sfb)}[f(\sfa) Z(\sfa) \mid \sfa|_{\Lambda'} = \sfa_1, \omega|_{E'} = \omega_1]}{\Psi_{\Lambda, \beta}^{(\xi, \sfb)}[Z(\sfa) \mid \sfa|_{\Lambda'} = \sfa_1, \omega|_{E'} = \omega_1]}.
    \end{equation*}
By taking the logarithmic derivative of $Z(\sfa)$ as in \eqref{eq: logarithmic derivative}, we deduce that $Z(\sfa)$ is an increasing function of $\sfa$, so by Proposition \ref{prop: FKG},
    \begin{equation}
        \label{eq: weak monotonicity abs}
        \Psi_{\Lambda, \beta}^{(\xi, \sfb)}[f(\sfa) \mid \sfa|_{\Lambda'} = \sfa_2, \omega|_{E'} = \omega_2] \geq \Psi_{\Lambda, \beta}^{(\xi, \sfb)}[f(\sfa) \mid \sfa|_{\Lambda'} = \sfa_1, \omega|_{E'} = \omega_1].
    \end{equation}
    To prove the full stochastic domination, let $g: (\mathbb{R}^{+})^{\overline{\Lambda}} \times \{0, 1\}^{\overline{E}(\Lambda)} \rightarrow \mathbb{R}$ be an increasing function and set $f(\sfa) = \phi_{G', \beta, \sfa}^{\xi_2}[g_{\sfa}]$, where $g_{\sfa}(\omega) = g(\sfa, \omega)$. By monotonicity properties of the standard random cluster model, $f$ is an increasing function of $\sfa$, so applying \eqref{eq: weak monotonicity abs} we obtain
    \begin{align*}
        \Psi_{\Lambda, \beta}^{(\xi, \sfb)}[g(\sfa, \omega) \mid \sfa|_{\Lambda'} = \sfa_2, \omega|_{E'} = \omega_2] &= \Psi_{\Lambda, \beta}^{(\xi, \sfb)}[\phi^{\xi_2}_{G', \beta. \sfa}[g_{\sfa}] \mid \sfa|_{\Lambda'} = \sfa_2, \omega|_{E'} = \omega_2]\\
        &\geq \Psi_{\Lambda, \beta}^{(\xi, \sfb)}[\phi^{\xi_2}_{G', \beta. \sfa}[g_{\sfa}] \mid \sfa|_{\Lambda'} = \sfa_1, \omega|_{E'} = \omega_1]\\
        &\geq \Psi_{\Lambda, \beta}^{(\xi, \sfb)}[\phi^{\xi_1}_{G', \beta. \sfa}[g_{\sfa}] \mid \sfa|_{\Lambda'} = \sfa_1, \omega|_{E'} = \omega_1]\\
        &=\Psi_{\Lambda, \beta}^{(\xi, \sfb)}[g(\sfa, \omega) \mid \sfa|_{\Lambda'} = \sfa_1, \omega|_{E'} = \omega_1],
    \end{align*}
    where in the third line we have used monotonicity in boundary conditions for the standard random cluster model.
\end{proof}

\begin{remark}
    We could allow the random cluster measure to be defined on configurations $(\sfa, \omega) \in (\mathbb{R}^{+})^{\overline{\Lambda}} \times \{0,1\}^{\overline{E}(\Lambda) \cup E'}$, where $E'$ consists of some edges $xy$ outside $\overline{E}(\Lambda)$ for which $p_{xy}$ takes the deterministic value given by the boundary conditions. Then by the domain Markov property, weak monotonicity of $\Psi_{\Lambda, \beta}^{(\xi, \sfb)}$ is equivalent to monotonicity in boundary conditions for this class of measures, which explains why the proof of Proposition \ref{prop: weak monotonicity} closely resembles that of \cite[Proposition 4.10]{well_behaved}.
\end{remark}

Having completed the proof of weak monotonicity, we now state two further correlation inequalities which will be useful in later proofs.
In the next proposition, we show that the FKG inequality holds for the measure $\Psi_{\Lambda, \beta}^{(\xi, \sfb)}$ conditional on $\omega_{xy} = 1$, a case which is not covered by Proposition \ref{prop: FKG}. 

\begin{proposition}
    \label{prop: open edge FKG}
    For any edge $xy \in \overline{E}(\Lambda)$, $\Psi_{\Lambda, \beta}^{(\xi, \sfb)}[\ \cdot \mid \omega_{xy} = 1]$ satisfies the FKG inequality.
\end{proposition}

\begin{proof}
    By Proposition \ref{prop: weak monotonicity} and Remark \ref{remark: extra vertex removal}, we know that $\Psi_{\Lambda, \beta}^{(\xi, \sfb)}[\ \cdot \mid \sfa_x = s, \omega_{xy} = 1]$ satisfies the FKG inequality for any $s \in \mathbb{R}^{+}$, which proves the case when $xy$ has one endpoint outside $\Lambda$. To remove the vertex $x$ from the conditioning in the case when both $x$ and $y$ are in $\Lambda$, we proceed as in the proof of Proposition \ref{prop: FKG} and first aim to show that for any $s \geq s'$ and any increasing function $f$ that depends only on $\sfa$,
    \begin{equation}
        \label{eq: open edge FKG 1}
        \Psi_{\Lambda, \beta}^{(\xi, \sfb)}[f(\sfa) \mid \sfa_x =s, \omega_{xy} = 1]
        \geq \Psi_{\Lambda, \beta}^{(\xi, \sfb)}[f(\sfa) \mid \sfa_x =s', \omega_{xy} = 1].
    \end{equation}
    Let $G' = (\overline{\Lambda}, \overline{E}(\Lambda) \setminus \{xy\})$ and let $\xi^{xy}$ be the partition obtained from $\xi$ by merging the element containing $x$ with the element containing $y$. Also let $\sfa \Delta s \in (\mathbb{R}^{+})^{\overline{\Lambda}}$ be the configuration obtained from $\sfa$ by setting $\sfa_x$ equal to $s$. We then have
    \begin{equation}
        \label{eq: open edge FKG 2}
        \Psi_{\Lambda, \beta}^{(\xi, \sfb)}[f(\sfa) \mid \sfa_x =s, \omega_{xy} = 1]
        = \frac{\Psi_{\Lambda, \beta}^{(\xi, \sfb)}[f(\sfa) Z(\sfa) \mid \sfa_x =s', \omega_{xy} = 1]}{\Psi_{\Lambda, \beta}^{(\xi, \sfb)}[Z(\sfa) \mid \sfa_x =s', \omega_{xy} = 1]},
    \end{equation}
    where
    \begin{equation*}
        Z(\sfa) = \frac{p_{xy}(\sfa \Delta s) \sqrt{1- p_{xy}(\sfa)} Z_{G', \beta, \sfa \Delta s}^{\mathrm{Ising}, \xi^{xy}}}{p_{xy}(\sfa)\sqrt{1- p_{xy}(\sfa \Delta s)} Z_{G', \beta, \sfa}^{\mathrm{Ising}, \xi^{xy}}}.
    \end{equation*}
    By taking the logarithmic derivative as in \eqref{eq: logarithmic derivative 2}, $\frac{Z_{G', \beta, \sfa \Delta s}^{\mathrm{Ising}, \xi^{xy}}}{Z_{G', \beta, \sfa}^{\mathrm{Ising}, \xi^{xy}}}$ is an increasing function of $\sfa$ when $\sfa_x = s'$. We also have that
    \begin{equation*}
        \frac{p_{xy}(\sfa \Delta s) \sqrt{1- p_{xy}(\sfa)}}{p_{xy}(\sfa)\sqrt{1- p_{xy}(\sfa \Delta s)}} = \frac{e^{\beta s \sfa_y} - e^{- \beta s \sfa_y}}{e^{\beta s' \sfa_y} - e^{-\beta s' \sfa_y}}
    \end{equation*}
    is an increasing function of $\sfa_y$. Hence $Z$ is an increasing function, and applying the FKG inequality in \eqref{eq: open edge FKG 2} yields \eqref{eq: open edge FKG 1}, which (following the same steps as in \eqref{eq: FKG induction 0}, \eqref{eq: FKG induction 1}, and \eqref{eq: FKG induction 2}) implies that for any increasing functions $f, g$ that depend only on $\sfa$,
    \begin{equation}
        \label{eq: open edge FKG 3}
        \Psi_{\Lambda, \beta}^{(\xi, \sfb)} [f(\sfa) g(\sfa) \mid \omega_{xy} =1] \geq \Psi_{\Lambda, \beta}^{(\xi, \sfb)} [f(\sfa) \mid \omega_{xy} = 1]\Psi_{\Lambda, \beta}^{(\xi, \sfb)} [g(\sfa) \mid \omega_{xy} = 1].
    \end{equation}
    To extend \eqref{eq: open edge FKG 3} to the case when $f$ and $g$ depend on both $\sfa$ and $\omega$, we write
    \begin{equation*}
        \Psi_{\Lambda, \beta}^{(\xi, \sfb)}[f(\sfa, \omega) g(\sfa, \omega) \mid \omega_{xy} = 1] = \Psi_{\Lambda, \beta}^{(\xi, \sfb)}[ \phi_{G', \beta, \sfa}^{\xi^{xy}}[f_{\sfa} g_{\sfa}] \mid \omega_{xy} = 1],
    \end{equation*}
    where $f_{\sfa}(\omega) = f(\sfa, \omega)$, $g_{\sfa}(\omega) = g(\sfa, \omega)$,
    and use the FKG inequality of the standard random cluster model. Since $\phi_{G', \beta, \sfa}^{\xi^{xy}}[f_{\sfa}]$ and $ \phi_{G', \beta, \sfa}^{\xi^{xy}}[g_{\sfa}]$ are increasing functions of $\sfa$, we can then apply \eqref{eq: open edge FKG 3} to conclude.
\end{proof}

In the following lemma, we prove that the random cluster measure satisfies assumption \eqref{eq: zero conditioning} in the statement of the OSSS inequality, which roughly speaking states that closing a vertex is less costly than closing the edges around it. 

\begin{lemma}
    \label{Lemma: zero conditioning}
    For any finite $\Lambda \subset V$, any boundary conditions $(\xi, \sfb)$ on $\Lambda$, and any $\beta \geq 0$, $\Psi_{\Lambda, \beta}^{(\xi, \sfb)}$ satisfies \eqref{eq: zero conditioning}.
\end{lemma}

\begin{proof}
    Let $C_z$ be the event that all edges incident to $z$ are closed and let $\sfb \cup 0_z \in (\mathbb{R}^{+})^{z \cup \partial^{\mathrm{ext}} \Lambda}$ be equal to $\sfb$ on $\partial^{\mathrm{ext}} \Lambda$ and $0$ on $z$.
    Let $N(z, \xi)$ be the number of elements of $\xi$ that only contain vertices which are not in $\partial^{\mathrm{ext}} (\Lambda \setminus \{z\})$.
    Observe that
    \begin{align*}
        \Psi_{\Lambda, \beta}^{(\xi, \sfb)} [f(\omega) \mathbbm{1}_{C_z}]
        &= \frac{2^{N(z, \xi)} Z_{\Lambda \setminus \{z\}, \beta}^{(\xi, \sfb\cup0_z)}}{Z_{\Lambda, \beta}^{(\xi, \sfb)}} \int_{\sfa_z = 0}^{\infty} \Psi_{\Lambda \setminus \{z\} , \beta}^{(\xi, \sfb \cup 0_z)} 
        \left[ f(\omega) \prod_{x \sim z} \sqrt{1- p_{xz}(\sfa)} \right] \mathrm{d} \rho(\mathsf{a}_z).
    \end{align*}
    Since $\prod_{x \sim z} \sqrt{1- p_{xz}(\mathsf{a})}$ is a decreasing function (for fixed $\mathsf{a}_z$), we can apply the FKG inequality to get
    \begin{align*}
        \Psi_{\Lambda, \beta}^{(\xi, \sfb)} [f(\omega) \mathbbm{1}_{C_z}] &\leq
        \frac{2^{N(z, \xi)} Z_{\Lambda \setminus \{z\}, \beta}^{(\xi, \sfb\cup0_z)}}{Z_{\Lambda, \beta}^{(\xi, \sfb)}} \Psi_{\Lambda \setminus \{z\} , \beta}^{(\xi, \sfb \cup 0_z)}[f(\omega)] \int_{\mathsf{a}_z = 0}^{\infty} \Psi_{\Lambda \setminus \{z\} , \beta}^{(\xi, \sfb \cup 0_z)} 
        \left[ \prod_{x \sim z} \sqrt{1- p_{xz}(\mathsf{a})} \right] \mathrm{d} \rho(\mathsf{a}_z)\\
        &= \Psi_{\Lambda \setminus \{z\} , \beta}^{(\xi, \sfb \cup 0_z)}[f(\omega)] \Psi_{\Lambda, \beta}^{(\xi, \sfb)}[C_z],
    \end{align*}
    hence
    \begin{equation*}
        \Psi_{\Lambda, \beta}^{(\xi, \sfb)}[f(\omega) \mid C_z] \leq \Psi_{\Lambda \setminus \{z\} , \beta}^{(\xi, \sfb \cup 0_z)}[f(\omega)] = \Psi_{\Lambda, \beta}^{(\xi, \sfb)}[f(\omega) \mid \mathsf{a}_z = 0],
    \end{equation*}
    where the last equality follows from the domain Markov property \eqref{eq: DMP}.
\end{proof}

\section{Non-triviality of phase transition}\label{sec:non-triviality}

In this section, we show that the phase transition on a graph $G$ of bounded degree is non-trivial if the phase transition for Bernoulli percolation is non-trivial. To do this, we will use the coupling between the spin model and the random cluster representation, in particular the fact that $\beta_c = \inf\{\beta \geq 0 : \, \Psi_{\beta}^{1}[o \longleftrightarrow \infty] > 0\}$ by Remark \ref{remark: infinite coupling}. We write $\mathbb{P}_p$ to denote the law of Bernoulli bond percolation, where each edge is open with probability $p$ independently of other edges, and write $o \longleftrightarrow \infty$ for the event that the cluster containing $o$ is infinite. Let $p_c$ denote the corresponding percolation threshold, i.e.\ $p_c=\inf\{p\in [0,1]: \, \mathbb{P}_p[o\longleftrightarrow\infty]>0\}$.
We denote the vertex and edge marginals of $\Psi_{\beta}^{1}$ by $\chi_{\beta}^{1}$ and $\Phi_{\beta}^{1}$ respectively. 

\begin{proposition}\label{prop:non-trivial}
Let $G=(V,E)$ be a connected graph of bounded degree. Then $\beta_c>0$. Furthermore, if $p_c<1$, then $\beta_c<\infty$. In particular, $\beta_c\in (0,\infty)$ on $\mathbb{Z}^d$ for every $d\geq 2$, and more generally, on every transitive graph of superlinear growth.
\end{proposition}
\begin{proof}
We first show that $\beta_c>0$, and for that we use the fact that $p_c>0$ on graphs of bounded degree. Consider some $\beta\in (0,1)$ which we will choose below to be small enough. By regularity, there exist constants $a,b>0$ such that for every $\beta\in (0,1)$, the absolute value field $\chi^1_{\beta}$ is stochastically dominated by the product measure $\nu_{a,b}$ with density proportional to $\prod_{x\in V} e^{at_x^2}\mathbbm{1}_{t_x\geq b}\mathrm{d}\rho(t_x-b)$. We define a percolation model $\mathbb{Q}$ as follows. We first consider some constant $M>0$ to be defined and a field $\sfa\sim \nu_{a,b}$. For every edge $xy \in E$ with $\sfa_x \geq M$ or $\sfa_y \geq M$, we set $\omega_{xy} = 1$. 
If $\sfa_x,\sfa_y< M$, then set $\omega_{xy}=1$ with probability $r=1-e^{-2\beta M^2}$, independently of all other edges. Note that $\Phi^1_{\beta}$ is stochastically dominated by $\mathbb{Q}$. Indeed, this follows from the fact that  $\Psi^1_{\beta}[\omega_{xy}=1 \mid \sfa, \omega|_{E\setminus \{xy\}}]\leq p_{xy}(\sfa)$, and $p_{xy}(\sfa)\leq r$ whenever $\sfa_x,\sfa_y< M$. Now $\mathbb{Q}$ is a $1$-dependent bond percolation measure. By \cite{LSS97}, it is stochastically dominated by $\mathbb{P}_p$ for some $p\in (0,1)$. By choosing $M$ to be large enough and then $\beta>0$ to be small enough, we can ensure that $\mathbb{Q}[\omega_e=1]$ is small enough that $p<p_c$. This implies that $\Psi^1_{\beta}[o\longleftrightarrow\infty]=0$, hence $\beta_c\geq \beta>0$.

We now show that if $p_c<1$, then $\beta_c<\infty$. By monotonicity in $\beta$, $\chi^1_{\beta}$ stochastically dominates $\chi^1_0$. We define a percolation model $\mathbb{Q}'$ as follows. We first consider an $\varepsilon>0$ such that $\rho([\varepsilon,\infty))>0$ and a field $\sfa\sim \chi^1_0$. For every $x\in V$ such that $\sfa_x<\varepsilon$, we set $\omega_{xy}=0$ for every neighbour $y$ of $x$, and for every $xy\in E$ such that $\sfa_x,\sfa_y\geq \varepsilon$, we set $\omega_{xy}=1$ with probability $\frac{1-e^{-2\beta \varepsilon^2}}{1+e^{-2\beta \varepsilon^2}}$ independently of all other edges. Using the fact that 
\[\Psi^1_{\beta}[\omega_{xy}=1 \mid \sfa, \omega|_{E\setminus \{xy\}}] \geq \frac{p_{xy}(\sfa)}{2 - p_{x,y}(\sfa)} =\frac{1-e^{-2 \beta \sfa_x \sfa_y}}{1+e^{-2 \beta \sfa_x \sfa_y}},\] 
it follows as above that $\Phi^1_{\beta}$ stochastically dominates $\mathbb{Q}'$. Now $\mathbb{Q}'$ is a $1$-dependent bond percolation measure. By \cite{LSS97}, it stochastically dominates $\mathbb{P}_p$ for some $p\in (0,1)$. By choosing $\beta$ to be large enough and $\varepsilon$ to be small enough, we can ensure that $\mathbb{Q}'[\omega_e=1]$ is close enough to $1$ that $p>p_c$. This implies that $\Psi^1_{\beta}[o\longleftrightarrow\infty]>0$, hence $\beta_c\leq \beta$.

The fact that $p_c<1$ on every transitive graph of superlinear growth follows from \cite{DGRSY20,EST25}. This implies that $\beta_c\in (0,\infty)$ on these graphs, in particular on $\mathbb{Z}^d$ for every $d\geq 2$, and completes the proof.
\end{proof}

\section{Proof of OSSS inequality}
\label{section: OSSS}
The OSSS inequality was originally proved for product measures \cite{OSSS}, and has subsequently been extended to the random cluster \cite{sharpness_annals} and dilute random cluster \cite{Blume-Capel} models.
In this section, we prove Theorem \ref{Thm: OSSS}, which gives an OSSS inequality for measures satisfying the monotonicity properties of Section \ref{Section: monotonicity}.

Before proving Theorem \ref{Thm: OSSS}, we first introduce some notation and definitions. Let $\Lambda$ be a finite subset of $V$ and let $n = |\Lambda \cup \overline{E}(\Lambda)|$. Define $\overrightarrow{\Lambda}$ to be the set of sequences of elements in $\Lambda \cup \overline{E}(\Lambda)$ where each element appears exactly once and for each edge $xy \in \overline{E}(\Lambda)$, both endpoints $x$ and $y$ appear before $xy$ (if one endpoint is outside $\Lambda$, then we only require that the other endpoint appears before $xy$). For an $n-$tuple $x = (x_1, \ldots, x_n)$ and $t \leq n$, we write $x_{[t]} =(x_1, \ldots, x_t)$ and $(\sfa, \omega)_{x_{[t]}} = (\sfa|_{\Lambda \cap x_{[t]}}, \omega|_{\overline{E}(\Lambda) \cap x_{[t]}})$.

\begin{definition}
    A \textit{decision tree} on $\Lambda$ is a pair $T = (x_1, (\psi_t)_{2 \leq t \leq n})$, where $x_1 \in \Lambda \cup \overline{E}(\Lambda)$ and for $2 \leq t \leq n$, $\psi_t$ is a function that takes as input $x_{[t-1]}$ and  $(\sfa, \omega)_{x_{[t-1]}}$ and returns an element $x_t \in \Lambda \cup \overline{E}(\Lambda) \setminus \{x_1, \ldots, x_{t-1}\}$.
    \begin{itemize}
    \item
    We say that $T$ is \textit{admissible} if it reveals edges only after any endpoints in $\Lambda$ have been revealed, so that the resulting sequence $x_1, \ldots, x_n$ is in $\overrightarrow{\Lambda}$.
    \item
    For a function $f: \{0, 1\}^{\overline{E}(\Lambda)} \rightarrow [0,1]$, we define
    \[
    \tau_{f, T}(\sfa, \omega) = \min \{t \geq 1: \forall \omega' \in \{0, 1\}^{\overline{E}(\Lambda)}, \omega'_{x_{[t]}} = \omega_{x_{[t]}} \Rightarrow f(\omega') = f(\omega)\}.
    \]
    This is the first time at which the value of $f$ is determined by the decision tree $T$.
    \item
    For $X \in \Lambda \cup \overline{E}(\Lambda)$, we write $R_{f, T}(X)$ for the event that there exists $t \leq \tau_{f, T}(\sfa, \omega)$ such that $x_t = X$. For $xy \in \overline{E}(\Lambda)$, we define
    \begin{align*}
         \delta_{xy}(f, T) = \mu[R_{f,T}(xy)] + \frac{\mu[R_{f,T}(x)]}{\mu[C_x]} + \frac{\mu[R_{f,T}(y)]}{ \mu[C_y]},
    \end{align*}
    where $C_x = \{\omega_{xy} = 0 \ \forall \, y \sim x \}$ and we interpret $\frac{\mu[R_{f, T}(x)]}{\mu[C_x]}$ as $0$ if $x \notin \Lambda$.
    \end{itemize}
\end{definition}

We now show how we can construct $(\sfa, \omega)$ distributed according to $\mu$ on $(\mathbb{R}^{+})^{\Lambda} \times \{0,1\}^{\overline{E}(\Lambda)}$ from i.i.d uniform random variables, similarly to \cite[Lemma 2.1]{sharpness_annals}.
For $u \in [0,1]^{n}$, $x \in \overrightarrow{\Lambda}$, define $F_{x}^{\mu}(u) = (\tilde{\sfa}, \tilde{\omega})$ inductively for $1 \leq t \leq n$ according to the following rules.
If $x_t \in \overline{E}(\Lambda)$, then we set
\begin{equation*}
\tilde{\omega}_{x_t} = \begin{cases}
        1 & \text{ if } u_t \geq \mu[\omega_{x_t} = 0 \mid  (\sfa, \omega)_{x_{[t-1]}}= (\tilde{\sfa}, \tilde{\omega})_{x_{[t-1]}}], \\
        0 & \text{ otherwise,}
    \end{cases}
\end{equation*}
and if $x_t \in \Lambda$, we set $\tilde{\sfa}_{x_t} = Q(u_t)$, where $Q$ is the quantile function given by
\begin{equation*}
    Q(r)=\inf\{\lambda\in \mathbb{R}: \, \mu[\sfa_{x_t} \leq \lambda\mid ( \sfa, \omega)_{x_{[t-1]}}= (\tilde{\sfa}, \tilde{\omega})_{x_{[t-1]}}]\geq r\}.
\end{equation*}
In the proof of Theorem \ref{Thm: OSSS}, we will consider $F_{\bfx}^{\mu}$, where $\bfx$ is a random exploration determined by the decision tree.
\begin{lemma}
    \label{Lemma: F_x(U)}
    Let $U= (U_1, \ldots, U_n)$ be a sequence of i.i.d Unif([0,1]) random variables and let $\bfx$ be a random variable taking values in $\overrightarrow{\Lambda}$. Assume that for every $1 \leq t \leq n$, $U_t$ is independent of $(\bfx_{[t]}, U_{[t-1]})$. Then $F_{\bfx}^{\mu}(U)$ has law $\mu$.
\end{lemma}

\begin{proof}
    Let $x \in \overrightarrow{\Lambda}$, $(\sfa', \omega') \in (\mathbb{R}^{+})^{\Lambda} \times \{0,1\}^{\overline{E}(\Lambda)}$ and $(\tilde{\sfa}, \tilde{\omega}) = F_{\bfx}^{\mu}(U)$. Write $\mathbb{P}_{t, x} = \mathbb{P}_{t, x, \sfa', \omega'}$ for the law of $U$ and $\bfx$ conditioned on $\bfx_{[t]} = x_{[t]}$ and $(\tilde{\sfa}, \tilde{\omega})_{x_{[t]}} = (\sfa', \omega')_{x_{[t]}}$.
    We aim to prove by induction on $t$ that for any $t \in \{0, \ldots, n\}$,
    \begin{equation}
    \label{eq: construction lemma hypothesis}
    \begin{aligned}
    \mathrm{d}\mathbb{P}_{t, x}[(\tilde{\sfa}, \tilde{\omega}) = (\sfa', \omega')] = \mathrm{d}\mu[(\sfa, \omega) = (\sfa', \omega') \mid (\sfa, \omega)_{x_{[t]}} = (\sfa', \omega')_{x_{[t]}}] \\\text{ for all } x \in \Lambda, (\sfa', \omega') \in (\mathbb{R}^{+})^{\Lambda} \times \{0,1\}^{\overline{E}(\Lambda)}.
    \end{aligned}
    \end{equation}
    The statement of the lemma follows by taking $t=0$ in \eqref{eq: construction lemma hypothesis}.
    To prove \eqref{eq: construction lemma hypothesis}, note that the case $t = n$ is trivial. Now assume \eqref{eq: construction lemma hypothesis} is true for $t$. By decomposing over the value of $\bfx_t \in \overline{E}(\Lambda) \cup \Lambda$, we have
    \begin{equation}
    \label{eq: construction lemma induction}
    \begin{aligned}
        \mathrm{d} \mathbb{P}_{t-1, x}[(\tilde{\sfa}, \tilde{\omega}) = (\sfa', \omega')] =\\ \sum_{y \in \overline{E}(\Lambda) \setminus \{x_1, \ldots, x_{t-1} \}} \mathrm{d}\mathbb{P}_{t, x^{(t, y)}}[(\tilde{\sfa}, \tilde{\omega}) = (\sfa', \omega')] \mathbb{P}_{t-1, x}[ \tilde{\omega}_{y} = \omega'_y \mid \bfx_t = y] \mathbb{P}_{t-1, x}[\bfx_t = y]\\
        + \sum_{y \in \Lambda \setminus \{x_1, \ldots, x_{t-1}\}} \mathrm{d}\mathbb{P}_{t, x^{(t, y)}}[(\tilde{\sfa}, \tilde{\omega}) = (\sfa', \omega')] \mathrm{d}\mathbb{P}_{t-1, x}[ \tilde{\sfa}_{y} = \sfa'_y \mid \bfx_t = y] \mathbb{P}_{t-1, x}[\bfx_t = y],
    \end{aligned}
    \end{equation}
    where $(x^{(t, y)})_t = y$ and $(x^{(t, y)})_i = x_i$ for $i \neq t$.
    If $y \in \overline{E}(\Lambda)$, then from the definition of $\tilde{\omega}$,
    \begin{equation}
    \label{eq: construction lemma edges}
    \begin{aligned}
        \mathbb{P}_{t-1, x}[\tilde{\omega}_y = 1 \mid \bfx_t = y] &= \mathbb{P}_{t-1, x}[U_t \geq \mu[\omega_y = 0 \mid (\sfa, \omega)_{\bfx_{[t-1]}} = (\tilde{\sfa}, \tilde{\omega})_{\bfx_{[t-1]}}] \mid \bfx_t = y]\\
        &= \mu[\omega_y = 1 \mid (\sfa, \omega)_{x_{[t-1]}} = (\sfa', \omega')_{x_{[t-1]}}],
    \end{aligned}
    \end{equation}
    where we have used that $U_t$ is independent of $(\bfx_{[t]}, U_{[t-1]})$.
    If $y \in \Lambda$, then by definition of $\tilde{\sfa}$ and independence of $U_t$ from $(\bfx_{[t]}, U_{[t-1]})$, we get
    \begin{equation}
        \label{eq: construction lemma vertices}
        \mathrm{d}\mathbb{P}_{t-1, x}[ \tilde{\sfa}_{y} = \sfa'_y \mid \bfx_t = y] = \mathrm{d}\mu[\sfa_y = \sfa'_y \mid (\sfa, \omega)_{x_{[t-1]}} = (\sfa', \omega')_{x_{[t-1]}}].
    \end{equation}
    Substituting \eqref{eq: construction lemma edges} and \eqref{eq: construction lemma vertices} in \eqref{eq: construction lemma induction} and using the inductive hypothesis \eqref{eq: construction lemma hypothesis} for $\mathbb{P}_{t, x^{(t, y)}}[(\tilde{\sfa}, \tilde{\omega}) = (\sfa', \omega')]$ yields
    \begin{align*}
        &\mathrm{d} \mathbb{P}_{t-1, x}[(\tilde{\sfa}, \tilde{\omega}) = (\sfa', \omega')]\\
        &= \mathrm{d} \mu[(\sfa, \omega) = (\sfa', \omega') \mid (\sfa, \omega)_{x_{[t-1]}} = (\sfa', \omega')_{x_{[t-1]}}] \sum_{y \in (\Lambda \cup \overline{E}(\Lambda))\setminus \{x_1, \ldots, x_{t-1}\}} \mathbb{P}_{t-1, x}[\bfx_t = y]\\
        &= \mathrm{d} \mu[(\sfa, \omega) = (\sfa', \omega') \mid (\sfa, \omega)_{x_{[t-1]}} = (\sfa', \omega')_{x_{[t-1]}}].
    \end{align*}
\end{proof}

We now prove the OSSS inequality of Theorem~\ref{Thm: OSSS} which plays a crucial role in proving Theorem~\ref{thm: main theorem}. We follow a strategy similar to that of \cite{sharpness_annals}. The main challenge compared to the OSSS inequality for monotonic measures arises from the fact that we now explore both vertices and edges. The contribution of edges can be analysed as in \cite{sharpness_annals}. To handle the contribution of vertices, we use our assumption \eqref{eq: zero conditioning} to get an upper bound that involves covariances of the form $\mathrm{Cov}_{\mu} (f, \omega_{xy})$ rather than covariances of $f$ with a function of $\sfa_x$, at the expense of introducing the term $1/\mu[C_x]$. Finally, since we are exploring both edges and vertices, the revealment naturally includes both the probability of exploring an edge and the probability of exploring a vertex.

\begin{proof}[Proof of Theorem \ref{Thm: OSSS}]
    Let $U = (U_1, \ldots, U_n)$ and $V = (V_1, \ldots, V_n)$ be two independent sequences of i.i.d Unif([0,1]) random variables. Write $\mathbb{P}$ for the coupling between these variables and $\mathbb{E}$ for its expectation.
    Define $\mathbf{x} \in \overrightarrow{\Lambda}$, $\tilde{\omega} \in \{0,1\}^{\overline{E}(\Lambda)}$ and $\tilde{\sfa} \in (\mathbb{R}^{+})^{\Lambda}$ inductively as follows.
    We determine $\mathbf{x}$ using the decision tree, so that
    \begin{equation*}
        \mathbf{x}_t =
        \begin{cases}
             x_1 &  \text{ if } t = 1,\\
             \psi_t(\mathbf{x}_{[t-1]}, (\tilde{\sfa}, \tilde{\omega})_{\mathbf{x}_{[t-1]}}) & \text{ if } t \geq 2,
        \end{cases}
    \end{equation*}
    where $(\tilde{\sfa}, \tilde{\omega})=F_{\mathbf{x}}^{\mu}(U)$.
    For $0 \leq t \leq n$, let $W^t = (V_1, \ldots, V_t, U_{t+1}, \ldots, U_{\tau}, V_{\tau+1}, \ldots, V_n)$, where $\tau = \tau_{f, T}(\tilde{\sfa}, \tilde{\omega})$, and let $(\sfa^t, \omega^t) = F^{\mu}_{\bfx}(W^t)$. By Lemma \ref{Lemma: F_x(U)}, $(\tilde{\sfa}, \tilde{\omega})$ has law $\mu$ and is $U-$measurable, while
    $(\sfa^{n}, \omega^{n})$ has law $\mu$ and is independent of $U$.
    Therefore, using that $f$ takes values in $[0, 1]$ we have
    \begin{align*}
        \mathrm{Var}_{\mu}(f) \leq \frac{1}{2} \mu[|f - \mu[f]|]
        = \frac{1}{2} \mathbb{E}\Big[\big| \mathbb{E}[f(\tilde{\omega}) - f(\omega^{n}) \mid U]\big|\Big]
        \leq \frac{1}{2} \mathbb{E}[|f(\tilde{\omega}) - f(\omega^{n})|].
    \end{align*}
    Since $f(\tilde{\omega}) = f(\omega^{0})$ because $\tilde{\omega}$ agrees with $\omega^{0}$ up to time $\tau$, we obtain
    \begin{align*}
        \mathrm{Var}_{\mu}(f) &\leq \frac{1}{2} \mathbb{E}[|f(\omega^{0}) - f(\omega^{n})|]
        \leq \frac{1}{2} \sum_{t=1}^{n} \mathbb{E}[|f(\omega^{t}) - f(\omega^{t-1})|]\\
        &= \frac{1}{2} \sum_{t=1}^{n} \mathbb{E}\left[|f(\omega^{t}) - f(\omega^{t-1})| \mathbbm{1}_{\{t \leq \tau\}}\right],
    \end{align*}
    where we have used that $\omega^{t} = \omega^{t-1}$ for any $t > \tau$. Now since $\bfx_t$ and the event $\{t \leq \tau\}$ are $U_{[t-1]}$-measurable,
    \begin{equation}
    \label{eq: telescoping sum}
    \begin{aligned}
        \mathrm{Var}_{\mu}(f) &\leq \frac{1}{2} \sum_{xy \in \overline{E}(\Lambda)} \sum_{t=1}^{n} \mathbb{E}\left[\mathbb{E}\left[|f(\omega^{t}) - f(\omega^{t-1})| \mid U_{[t-1]}\right] \mathbbm{1}_{\{t \leq \tau, \bfx_t = xy\}}\right]\\
        &+ \frac{1}{2} \sum_{y \in \Lambda} \sum_{t=1}^{n} \mathbb{E}\left[\mathbb{E}\left[|f(\omega^{t}) - f(\omega^{t-1})| \mid U_{[t-1]}\right] \mathbbm{1}_{\{t \leq \tau, \bfx_t = y\}}\right].
    \end{aligned}
    \end{equation}
    We first tackle the terms in \eqref{eq: telescoping sum} where $\bfx_t$ is an edge, aiming to show that on the event $\{t \leq \tau, \bfx_t = xy \in \overline{E}(\Lambda)\}$,
    \begin{equation}
        \label{eq: OSSS edge terms}
        \mathbb{E}\left[|f(\omega^{t}) - f(\omega^{t-1})| \mid U_{[t-1]}\right] \leq 2 \mathrm{Cov}_\mu (f, \omega_{xy}).
    \end{equation}
    The inequality \eqref{eq: OSSS edge terms} follows as in the proof of \cite[Theorem 1.1]{sharpness_annals}, but we repeat the argument here for the reader's convenience.
    First observe that on the event $\{t \leq \tau, \bfx_t = xy \in \overline{E}(\Lambda)\}$, if $\omega^{t}_{xy} = \omega^{t-1}_{xy}$ then $(\sfa^{t}, \omega^{t}) = (\sfa^{t-1}, \omega^{t-1})$. Together with the fact that $f$ is increasing, this implies that 
    \begin{align}
        \label{eq: f(omega^t) - f(omega^t-1)}
        |f(\omega^t) - f(\omega^{t-1})| &= (f(\omega^t) - f(\omega^{t-1}))(\omega^t_{xy} - \omega^{t-1}_{xy})\\
        \nonumber
        &= f(\omega^{t-1}) \omega^{t-1}_{xy} + f(\omega^{t})\omega^{t}_{xy} - f(\omega^{t-1}) \omega^{t}_{xy} - f(\omega^{t}) \omega^{t-1}_{xy}.
    \end{align}
    Also notice that Lemma \ref{Lemma: F_x(U)} implies that for any measurable function $g: (\mathbb{R}^{+})^{\Lambda} \times \{0, 1\}^{\overline{E}(\Lambda)} \rightarrow \mathbb{R}$ and any $k \leq t \leq n$,
    \begin{equation}
        \label{eq: applying F_x(U) lemma}
        \mathbb{E}[ g(\sfa^{t}, \omega^{t}) \mid U_{[k]}] = \mu[g(\sfa, \omega)].
    \end{equation}
    Applying \eqref{eq: applying F_x(U) lemma} to $g(\sfa, \omega) = f(\omega) \omega_{xy}$ gives
    \begin{equation}
        \label{eq: OSSS same t terms}
        \mathbb{E}[f(\omega^{t-1})\omega^{t-1}_{xy} \mid U_{[t-1]}] = \mu[f(\omega) \omega_{xy}] = \mathbb{E}[f(\omega^{t})\omega^{t}_{xy} \mid U_{[t-1]}].
    \end{equation}
    Since $\mu$ is weakly monotonic and $T$ is an admissible decision tree, for fixed $U_{[n]}$ and $s$, $\omega^{s}$ is an increasing function of $V$. Since $f(\omega)$ and $\omega_{xy}$ are increasing functions of $\omega$, we deduce that $f(\omega^{t-1})$ and $\omega^{t}_{xy}$ are increasing functions of $V$. 
    The FKG inequality applied to the i.i.d random variables $V_1, \ldots, V_n$ gives
    \begin{equation*}
        \mathbb{E}[f(\omega^{t-1}) \omega^{t}_{xy} \mid U_{[n]}] \geq \mathbb{E}[f(\omega^{t-1}) \mid U_{[n]}] \mathbb{E}[\omega^{t}_{xy} \mid U_{[n]}].
    \end{equation*}
    Taking the expectation with respect to $\mathbb{E}[\ \cdot \mid U_{[t-1]}]$ gives
    \begin{align}
        \label{eq: OSSS different t terms}
        \mathbb{E}[f(\omega^{t-1}) \omega^{t}_{xy} \mid U_{[t-1]}] &\geq \mathbb{E}\left[ \mathbb{E}[f(\omega^{t-1}) \mid U_{[n]}] \mathbb{E}[\omega^{t}_{xy} \mid U_{[n]}] \ \big| \ U_{[t-1]} \right]\\
        \nonumber
        &= \mathbb{E}[f(\omega^{t-1}) \mid U_{[t-1]}] \mathbb{E}[\omega^{t}_{xy} \mid U_{[t-1]}] = \mu[f(\omega)] \mu[\omega_{xy}],
    \end{align}
    where we have used that $\mathbb{E}[\omega^{t}_{xy} \mid U_{[n]}]$ is $U_{[t-1]}$-measurable and \eqref{eq: applying F_x(U) lemma}.
    Similarly, $f(\omega^{t})$ and $\omega^{t-1}_{xy}$ are increasing functions of $V$, so applying the FKG inequality and taking expectation with respect to $\mathbb{E}[\ \cdot \mid U_{[t]}]$ gives
    \begin{align*}
        \mathbb{E}[f(\omega^{t}) \omega^{t-1}_{xy} \mid U_{[t]}] &\geq
        \mathbb{E} \left[ \mathbb{E}[f(\omega^{t}) \mid U_{[n]}] \mathbb{E}[\omega^{t-1}_{xy} \mid U_{[n]}] \big| U_{[t]} \right]\\
        &= \mathbb{E}[f(\omega^{t}) \mid U_{[t]}] \mathbb{E}[\omega^{t-1}_{xy} \mid U_{[t]}] = \mu[f(\omega)] \mathbb{E}[\omega^{t-1}_{xy} \mid U_{[t]}].
    \end{align*}
    Taking expectation with respect to $\mathbb{E}[\ \cdot \mid U_{[t-1]}]$ and using \eqref{eq: applying F_x(U) lemma}, we get
    \begin{equation*}
        \mathbb{E}[f(\omega^{t}) \omega^{t-1}_{xy} \mid U_{[t-1]}] \geq
        \mu[f(\omega)] \mathbb{E}[\omega^{t-1}_{xy} \mid U_{[t-1]}] = \mu[f(\omega)] \mu[\omega_{xy}].
    \end{equation*}
    Combining the above inequality with \eqref{eq: f(omega^t) - f(omega^t-1)}, \eqref{eq: OSSS same t terms} and \eqref{eq: OSSS different t terms} yields \eqref{eq: OSSS edge terms}.
    
    We now consider the terms in \eqref{eq: telescoping sum} where $\bfx_t$ is a vertex. On the event $\{t \leq \tau, \bfx_t = y \in \Lambda\}$ define 
    \[
    (\sfa^*, \omega^*) = F^{\mu}_{\bfx}(V_1, \ldots, V_{t-1}, 0, U_{t+1}, \ldots, U_{\tau}, V_{\tau+1}, \ldots V_n).
    \]
    Then weak monotonicity (which applies because $T$ is admissible) implies that $\omega^* \leq \omega^t, \omega^{t-1}$, so
    \begin{equation}\label{eq: trick}
    \begin{aligned}
        \mathbb{E}\left[|f(\omega^{t}) - f(\omega^{t-1})| \mid U_{[t-1]}\right]
        &\leq \mathbb{E}\left[|f(\omega^{t}) - f(\omega^*)| + |f(\omega^*) -f(\omega^{t-1})| \mid U_{[t-1]}\right]\\
        &= \mathbb{E}[f(\omega^{t}) + f(\omega^{t-1}) - 2 f(\omega^*) \mid U_{[t-1]}].
    \end{aligned}
    \end{equation}
    
    We claim that on the event $\{t \leq \tau, \bfx_t = y\}$, 
    \[
    \mathbb{E}[f(\omega^*) \mid U_{[t-1]}]\geq \mu[f(\omega) \mid \omega_{yz} = 0 \ \forall z \sim y].
    \]
    Indeed, by weak monotonicity, $(\sfa^*, \omega^*)\geq (\sfa',\omega')$, where 
    \[
    (\sfa', \omega') = F^{\mu'}_{\bfx}(V_1, \ldots, V_{t-1}, 0, U_{t+1}, \ldots, U_{\tau}, V_{\tau+1}, \ldots V_n)
    \]
    for $\mu'=\mu[\cdot \mid \sfa_y=0]$. Since $(\sfa', \omega')$ has law $\mu'$ on the event $\{t \leq \tau, \bfx_t = y\}$ by Lemma \ref{Lemma: F_x(U)}, the claim follows from our assumption that \eqref{eq: zero conditioning} holds. 
     
    Combining the claim with \eqref{eq: trick} and using \eqref{eq: applying F_x(U) lemma}, we get
    \begin{align*}
        \mathbb{E}\left[|f(\omega^{t}) - f(\omega^{t-1})| \mid U_{[t-1]}\right]
        &\leq 2 \mu[f(\omega)] - 2 \mu[f(\omega) \mid \omega_{yz} = 0 \ \forall z \sim y]\\
        &=\frac{2}{\mu[A_{y}^{\mathsf{c}}]} \mathrm{Cov}_\mu (f, \mathbbm{1}_{A_y}),
    \end{align*}
    where $A_y = C_y^{\mathsf{c}}$ is the event that $\omega_{yz} = 1$ for some $z \sim y$. Using that $\left(\sum_{z \sim y} \omega_{yz}\right) - \mathbbm{1}_{A_y}$ is an increasing function, we can apply the FKG inequality to get
    \[
    \sum_{z \sim y}\mathrm{Cov}_{\mu} (f, \omega_{yz})=\mathrm{Cov}_\mu (f, \mathbbm{1}_{A_y})+\mathrm{Cov}_\mu \left(f, \left(\sum_{z \sim y} \omega_{yz}\right) - \mathbbm{1}_{A_y}\right)\geq \mathrm{Cov}_\mu (f, \mathbbm{1}_{A_y}).
    \]
    Hence
    \begin{equation}
        \label{eq: OSSS vertex terms}
        \mathbb{E}\left[|f(\omega^{t}) - f(\omega^{t-1})| \mid U_{[t-1]}\right] \leq \frac{2}{\mu[A_{y}^{\mathsf{c}}]} \sum_{z \sim y} \mathrm{Cov}_{\mu} (f, \omega_{yz}).
    \end{equation}
    Combining \eqref{eq: telescoping sum}, \eqref{eq: OSSS edge terms} and \eqref{eq: OSSS vertex terms} and recalling that $\sum_{t = 1}^{n} \mathbb{P}[t \leq \tau, \bfx_t = X] = \mu[R_{f,T}(X)]$, we obtain
    \begin{align*}
        \mathrm{Var}_\mu (f) &\leq \sum_{xy \in \overline{E}(\Lambda)} \mu[R_{f,T}(xy)] \mathrm{Cov}_\mu (f, \omega_{xy}) 
        + \sum_{y \in \Lambda} \frac{\mu[R_{f,T}(y)]}{\mu[A_{y}^{\mathsf{c}}]} \sum_{z \sim y} \mathrm{Cov}_{\mu} (f, \omega_{yz})\\
        &\leq \sum_{xy \in \overline{E}(\Lambda)} \delta_{xy}(f, T) \mathrm{Cov}_{\mu} (f, \omega_{xy}). 
    \end{align*}
\end{proof}

\section{Subcritical sharpness for the random cluster model}\label{sec:sharpness random cluster}
In this section, we prove subcritical sharpness for the random cluster representation, which will be used in the next section to prove Theorem~\ref{thm: main theorem}.

Recall that $\beta_c = \inf\{\beta \geq 0 : \nu^{+}_{\beta}[\varphi_o] > 0\}$ is the critical point for the spin model and that $\beta_c>0$ by Proposition~\ref{prop:non-trivial}. Also note that by Remark \ref{remark: infinite coupling}, $\beta_c$ coincides with the critical point for the emergence of an infinite cluster under $\Psi^1_{\beta}$. 
We will obtain a subcritical sharpness result for the random cluster model, in the sense that for any $\beta < \beta_c$ and any boundary field $\sfb$ growing at most exponentially, the probability of the event that the origin is connected to $\partial \Lambda_n$ decays exponentially in $n$. Along the way, we also show the so-called mean-field lower bound for the percolation density in the supercritical regime.

\begin{theorem}
    \label{Thm: sharpness RC}
    Assume that $\sfb \in (\mathbb{R}^{+})^{V}$ is such that there exists $\lambda > 0$ with $\sfb_x \leq e^{\lambda d_G(o,x)}$ for all $x \in V$. For every $\beta < \beta_c$ there exists $c = c(\beta, \lambda) > 0$ such that for every $n \geq 0$  and any partition $\xi$ of $\partial^{\mathrm{ext}}\Lambda_n$,
    \begin{equation*}
        \Psi_{\Lambda_n, \beta}^{(\xi, \sfb)}[o \longleftrightarrow \partial \Lambda_n] \leq e^{-cn}.
    \end{equation*}
Moreover, there exists $c'>0$ such that for every $\beta\geq\beta_c$ close enough to $\beta_c$ we have 
\[
\Psi^1_{\beta}[o\longleftrightarrow \infty]\geq c'(\beta-\beta_c).
\]
\end{theorem}

To prove Theorem \ref{Thm: sharpness RC}, we will apply the OSSS inequality to a decision tree determining the event $o \longleftrightarrow \partial \Lambda_n$. We will see that this can be viewed as a differential inequality for the probability that $o \longleftrightarrow \partial \Lambda_n$, which allows us to apply \cite[Lemma 3.1]{sharpness_annals} to deduce sharpness.
We begin by calculating the derivative with respect to $\beta$. Below, if $\sfa_x \sfa_y=0$, we set $\frac{2\mathsf{a}_x \mathsf{a}_y \omega_{xy}}{p_{xy}(\mathsf{a})}=0$, since $\omega_{xy}=0$ in this case. We also write $\mathrm{Cov}_{\Lambda, \beta}^{(\xi, \sfb)}$ for the covariance with respect to the measure $\Psi_{\Lambda, \beta}^{(\xi, \sfb)}.$
\begin{proposition}
    \label{prop: beta derivative}
    For any function $f: (\mathbb{R}^{+})^{\overline{\Lambda}} \times \{0,1\}^{\overline{E}(\Lambda)}    \rightarrow \mathbb{R}$, we have
    \begin{equation*}
        \frac{\mathrm{d} \Psi_{\Lambda, \beta}^{(\xi, \sfb)}[f(\mathsf{a}, \omega)]}{\mathrm{d} \beta}
        = \sum_{xy \in \overline{E}(\Lambda)}  \mathrm{Cov}_{\Lambda, \beta}^{(\xi, \sfb)} \left(f, \frac{2\mathsf{a}_x \mathsf{a}_y \omega_{xy}}{p_{xy}(\mathsf{a})} - \sfa_x \sfa_y \right).
    \end{equation*}
\end{proposition}

\begin{proof}
    Set $r_{xy}(\sfa) = \sqrt{1 - p_{xy}(\sfa)} = e^{-\beta \sfa_x \sfa_y}$.
    For any edge $xy \in \overline{E}(\Lambda)$ with $\sfa_x \sfa_y \neq 0$, we have
    \begin{align*}
        \frac{\mathrm{d} }{\mathrm{d} \beta} \left( r_{xy}(\mathsf{a}) \left( \frac{p_{xy}(\mathsf{a})}{1 -p_{xy}(\mathsf{a})} \right)^{\omega_{xy}} \right) &=
        r_{xy} \left( \frac{p_{xy}}{1 -p_{xy}} \right)^{\omega_{xy}} \left( \frac{\omega_{xy}}{p_{xy}(1 - p_{xy})} \frac{\mathrm{d} p_{xy}}{\mathrm{d} \beta}+\frac{1}{r_{xy}} \frac{\mathrm{d}r_{xy}}{\mathrm{d}\beta} \right)
        \\
        &= r_{xy} \left( \frac{p_{xy}}{1 -p_{xy}} \right)^{\omega_{xy}} \left( \frac{2 \mathsf{a}_x \mathsf{a}_y \omega_{xy}}{p_{xy}}-\sfa_x\sfa_y\right).
    \end{align*}
    When $\sfa_x \sfa_y = 0$, it is still true that
        \begin{align*}
        \frac{\mathrm{d} }{\mathrm{d} \beta} \left( r_{xy}(\mathsf{a}) \left( \frac{p_{xy}(\mathsf{a})}{1 -p_{xy}(\mathsf{a})} \right)^{\omega_{xy}} \right)
        = r_{xy} \left( \frac{p_{xy}}{1 -p_{xy}} \right)^{\omega_{xy}} \left( \frac{2 \mathsf{a}_x \mathsf{a}_y \omega_{xy}}{p_{xy}}-\sfa_x\sfa_y\right)
    \end{align*}
    because both sides are $0$.
    The desired result follows from the product rule.
\end{proof}

We now relate the derivative to a sum of covariances with respect to the states of edges, which allows us to interpret the OSSS inequality as a differential inequality for $\Psi_{\Lambda_n, \beta}^{(\xi, \sfb)}[o \longleftrightarrow \partial \Lambda_n]$.
The first step is to lower bound the covariances appearing in Proposition~\ref{prop: beta derivative} by covariances that involve $\omega_{xy}$ and a function $g_{xy}$ defined by
\[
g_{xy}(\sfa):=(1- e^{-2 \beta \sfa_x \sfa_y})\mathbbm{1}_{\{2\beta\sfa_x \sfa_y < 1\}}+(1- e^{-1})\mathbbm{1}_{\{2\beta\sfa_x \sfa_y \geq 1\}}.
\]
This is done in Lemma \ref{lem: cov comparison} below. We then show in Lemma \ref{lem:cov comparison new} how this can be used to obtain a lower bound in terms of $\mathrm{Cov}_{\Lambda, \beta}^{(\xi, \sfb)}(f, \omega_{xy})$. 

\begin{lemma}\label{lem: cov comparison}
Let $\beta>0$ and let $\Lambda\subset V$ be finite with boundary conditions $(\xi, \sfb)$ on $\Lambda$. For every increasing function $f: (\mathbb{R}^{+})^{\overline{\Lambda}} \times \{0,1\}^{\overline{E}(\Lambda)}    \rightarrow \mathbb{R}$ and every $xy\in \overline{E}(\Lambda)$, we have
\begin{equation}
    \label{eq: cov comparison}
    \mathrm{Cov}_{\Lambda, \beta}^{(\xi, \sfb)} \left(f, \frac{2\sfa_x \sfa_y \omega_{xy}}{p_{xy}} - \sfa_x \sfa_y \right) \geq \frac{1}{2 \beta} \mathrm{Cov}_{\Lambda, \beta}^{(\xi, \sfb)}(f, g_{xy} \, \omega_{xy}).
\end{equation} 
\end{lemma}
\begin{proof}
Below we write $\mu=\Psi_{\Lambda, \beta}^{(\xi, \sfb)}$ to simplify the notation. It suffices to show that 
\begin{equation}
    \label{eq: cov >= 0}
    \mathrm{Cov}_{\mu} \left(f, \frac{h_{xy}\sfa_x \sfa_y \omega_{xy}}{p_{xy}} - \sfa_x \sfa_y \right) \geq 0,
\end{equation}    
where $h_{xy}=2-g_{xy}$.
Indeed, from \eqref{eq: cov >= 0}, we have
\begin{align*}
    \mathrm{Cov}_{\mu} \left(f, \frac{2\sfa_x \sfa_y \omega_{xy}}{p_{xy}} - \sfa_x \sfa_y \right) &\geq \mathrm{Cov}_{\mu} \left(f, g_{xy}\frac{\sfa_x \sfa_y \omega_{xy}}{p_{xy}}\right)\\
    &= \mathrm{Cov}_{\mu}\left( f, \frac{g_{xy}}{2 \beta} \omega_{xy} \right) + \mathrm{Cov}_{\mu}\left( f, g_{xy} \left( \frac{\sfa_x \sfa_y}{p_{xy}} - \frac{1}{2 \beta}\right) \omega_{xy} \right).
\end{align*}
Since $g_{xy}$ and $\frac{\sfa_x \sfa_y}{p_{xy}(\sfa)}$ are increasing functions of $\sfa$ and $\frac{\sfa_x \sfa_y}{p_{xy}(\sfa)}\geq \frac{1}{2\beta}$, the last covariance above is positive, hence \eqref{eq: cov comparison} holds.

We now turn to the proof of \eqref{eq: cov >= 0}. We have
    \begin{align*}
        \mathrm{Cov}_{\mu} \left(f, \frac{h_{xy} \sfa_x \sfa_y \omega_{xy}}{p_{xy}} - \sfa_x \sfa_y \right)
        &= \mu \left[ \sfa_x \sfa_y \left( \frac{h_{xy}\, \omega_{xy}}{p_{xy}} -1 \right)(f(\sfa,\omega) - \mu[f])  \right]\\
        &= \mu \left[ \sfa_x \sfa_y \mu\left[ \left( \frac{h_{xy}\, \omega_{xy}}{p_{xy}} -1 \right) (f(\sfa,\omega) - \mu[f]) \ \bigg| \ \sfa   \right] \right]\\
        &\geq \mu \left[ \sfa_x \sfa_y \left( \frac{h_{xy}}{p_{xy}} \mu[\omega_{xy} \mid \sfa] -1 \right) \mu \left[ f(\sfa,\omega) - \mu[f] \mid \sfa   \right] \right],
    \end{align*}
    where in the last inequality we used the FKG inequality for $\mu[\ \cdot \mid \sfa]$, which follows from Proposition \ref{prop: FKG}.
    
    We claim that $\sfa_x \sfa_y \left( \frac{h_{xy}}{p_{xy}} \mu[\omega_{xy} \mid \sfa] -1 \right)$ is non-negative and is an increasing function of $\sfa$.
    Using the claim, \eqref{eq: cov >= 0} follows by applying the FKG inequality again.
    To prove the claim, first note that 
    \begin{equation}
    \label{eq: finite-energy}
    \mu[\omega_{xy} \mid \sfa] = p_{xy} \mu[x \xleftrightarrow{\omega_{\{xy\}}} y \mid \sfa] + \frac{p_{xy}}{2 - p_{xy}} \mu[x \centernot{\xleftrightarrow{\omega_{\{xy\}}}} y \mid \sfa]\geq \frac{p_{xy}}{2-p_{xy}}\geq  \frac{p_{xy}}{h_{xy}},
    \end{equation}
    where $\omega_{\{xy\}}$ is the configuration obtained from $\omega$ by setting edge $xy$ to be closed. This implies that $\sfa_x \sfa_y \left( \frac{h_{xy}}{p_{xy}} \mu[\omega_{xy} \mid \sfa] -1 \right)$ is non-negative.
    Continuing from \eqref{eq: finite-energy}, we have
    \begin{align*}
        \frac{1}{p_{xy}} \mu[\omega_{xy} \mid \sfa] = 1 - \left(1 - \frac{1}{2 - p_{xy}} \right) \mu[x \centernot{\xleftrightarrow{\omega_{\{xy\}}}} y \mid \sfa].
    \end{align*}
    Since $1-\frac{1}{2-p_{xy}}$ and $\mu[x \centernot{\xleftrightarrow{\omega_{\{xy\}}}} y \mid \sfa]$ are non-negative and decreasing functions of $\sfa$, it follows that $\sfa_x \sfa_y \left( \frac{h_{xy}}{p_{xy}} \mu[\omega_{xy} \mid \sfa] -1 \right)$ is an increasing function of $\sfa$ when $2\beta\sfa_x \sfa_y \geq 1$. When $2\beta\sfa_x \sfa_y < 1$, the claim follows from writing
    \begin{align*}
        \sfa_x \sfa_y \left( \frac{h_{xy}}{p_{xy}} \mu[\omega_{xy} \mid \sfa] -1 \right) &= \sfa_x \sfa_y \left(\frac{2 - p_{xy}}{p_{xy}}\left( p_{xy} \mu[x \xleftrightarrow{\omega_{\{xy\}}} y \mid \sfa] + \frac{p_{xy}}{2 - p_{xy}} \mu[x \centernot{\xleftrightarrow{\omega_{\{xy\}}}} y \mid \sfa]\right) -1 \right)\\
        &= \sfa_x \sfa_y (1-p_{xy}) \mu[x \xleftrightarrow{\omega_{\{xy\}}} y \mid \sfa],
    \end{align*}
    and using that $\sfa_x \sfa_y (1-p_{xy})$ is an increasing function (for $2\beta\sfa_x \sfa_y < 1$).
\end{proof}

As a corollary, we get the following monotonicity in $\beta$.

\begin{corollary}\label{cor:monotonicity in beta}
Let $0\leq \beta_1\leq \beta_2$. Then for every increasing function $f: (\mathbb{R}^{+})^{\overline{\Lambda}} \times \{0,1\}^{\overline{E}(\Lambda)}    \rightarrow \mathbb{R}$,
\[
\Psi_{\Lambda, \beta_1}^{(\xi, \sfb)}[f]\leq \Psi_{\Lambda, \beta_2}^{(\xi, \sfb)}[f].
\]
\end{corollary}
\begin{proof}
From Proposition \ref{prop: beta derivative} and Lemma \ref{lem: cov comparison}, we have
\begin{equation*}
    \frac{\mathrm{d} \Psi_{\Lambda, \beta}^{(\xi, \sfb)}[f]}{\mathrm{d} \beta}
    = \sum_{xy \in \overline{E}(\Lambda)}  \mathrm{Cov}_{\Lambda, \beta}^{(\xi, \sfb)} \left(f, \frac{2\mathsf{a}_x \mathsf{a}_y \omega_{xy}}{p_{xy}} - \sfa_x \sfa_y \right) \geq \frac{1}{2\beta} \sum_{xy \in \overline{E}(\Lambda)}  \mathrm{Cov}_{\Lambda, \beta}^{(\xi, \sfb)} \left(f, g_{xy} \omega_{xy}\right).
\end{equation*}
Since $g_{xy} \omega_{xy}$ is an increasing function, $\mathrm{Cov}_{\Lambda, \beta}^{(\xi, \sfb)} \left(f, g_{xy} \omega_{xy}\right) \geq 0$ for any $xy \in \overline{E}(\Lambda)$ by the FKG inequality. The desired result follows.
\end{proof}

The caveat of Lemma~\ref{lem: cov comparison} is that the function $g_{xy}$ is close to $0$ when $\beta \sfa_{x}\sfa_y$ is close to $0$, which does not directly allow us to get a lower bound that involves the covariance with respect to $\omega_{xy}$. Nonetheless, $g_{xy}$ is on average bounded away from $0$, which we combine with the (conditional) FKG inequality to get the following result.

\begin{lemma}\label{lem:cov comparison new}
    Let $\beta > 0$. There exists $\varepsilon=\varepsilon(\beta) > 0$ that depends continuously on $\beta$ such that for every $\Lambda \subset V$ finite, every edge $xy \in \overline{E}(\Lambda)$, and every increasing function $f: (\mathbb{R}^{+})^{\overline{\Lambda}} \times \{0,1\}^{\overline{E}(\Lambda)}    \rightarrow \mathbb{R}$,
    \begin{equation*}
        \mathrm{Cov}_{\Lambda, \beta}^{(\xi, \sfb)} \left(f, \frac{2\sfa_x \sfa_y \omega_{xy}}{p_{xy}} - \sfa_x \sfa_y \right) \geq \varepsilon \mathrm{Cov}_{\Lambda, \beta}^{(\xi, \sfb)}(f, \omega_{xy}).
    \end{equation*}
\end{lemma}

\begin{proof}
Below we write $\mu = \Psi_{\Lambda, \beta}^{(\xi, \sfb)}$ to simplify the notation. By Lemma \ref{lem: cov comparison}, it suffices to prove that there exists $\varepsilon > 0$ such that $\mathrm{Cov}_{\mu}(f, (g_{xy} - \varepsilon)\omega_{xy}) \geq 0$. Let $\varepsilon > 0$ be a constant to be determined. Then
    \begin{align*}
        \mathrm{Cov}_{\mu}(f, (g_{xy} - \varepsilon)\omega_{xy}) &= \mu \left[(f(\sfa, \omega) - \mu[f]) (g_{xy} - \varepsilon) \omega_{xy}\right]\\
        &= \mu[\omega_{xy} = 1] \mu[(f(\sfa, \omega) - \mu[f]) (g_{xy} - \varepsilon) \mid \omega_{xy} = 1]\\
        &\geq \mu[\omega_{xy} = 1] \mu[f(\sfa, \omega) - \mu[f] \mid \omega_{xy} = 1] \mu[g_{xy} - \varepsilon \mid \omega_{xy} = 1],
    \end{align*}
    where the last line above follows from Proposition \ref{prop: open edge FKG} since $f(\sfa, \omega) - \mu[f]$ and $g_{xy} - \varepsilon$ are increasing functions.
    By the FKG inequality applied to $\mu$, we have $\mu[f(\sfa, \omega) - \mu[f] \mid \omega_{xy} = 1] \geq 0$. 
    Also note that we can choose $\varepsilon > 0$ small enough that $\mu[g_{xy} - \varepsilon \mid \omega_{xy} = 1] \geq 0$. Indeed, by the FKG inequality we have $\mu[g_{xy} \mid \omega_{xy} = 1]\geq \mu[g_{xy}]$, and the latter remains bounded away from $0$ since $\Psi_{\Lambda, 0}^{(w, \sfb)}[g_{xy}]$ remains bounded away from $0$ and we have monotonicity in $\beta$ by Corollary~\ref{cor:monotonicity in beta}. The desired result follows.
\end{proof}

The next step is to apply the OSSS inequality to certain well-chosen decision trees that determine the event $\{o \longleftrightarrow \partial \Lambda_n\}$. As in \cite[Lemma 8.2]{Blume-Capel} and \cite[Lemma 3.2]{sharpness_annals}, we consider a family of decision trees exploring the cluster of $\partial \Lambda_k$ for $k \in \{1, \ldots ,n\}$ and we then average over $k$.
\begin{lemma}
    \label{Lem: decision tree}
    Suppose $n \geq 1$ and $\Lambda \supset \Lambda_{2n}$. Let $\mu$ be a measure on $(\mathbb{R}^{+})^{\Lambda} \times \{0, 1\}^{\overline{E}(\Lambda)}$ that satisfies weak monotonicity and \eqref{eq: zero conditioning}. Then
    \begin{align*}
        \sum_{e \in \overline{E}(\Lambda)} \mathrm{Cov}_\mu(\mathbbm{1}_{\{o \longleftrightarrow \partial \Lambda_n\}}, \omega_{e}) \geq \frac{n \min_{x \in \Lambda_n} \mu[C_x]}{(4D+6) Q_n} \mu[o \longleftrightarrow \partial \Lambda_n] (1 - \mu[o \longleftrightarrow \partial \Lambda_n] ),
    \end{align*}
    where $Q_n \coloneq \max_{x \in \Lambda_n} \sum_{k=0}^{n-1} \mu[x \longleftrightarrow \partial \Lambda_k(x)]$.
\end{lemma}

\begin{proof}
    Let $f = \mathbbm{1}_{\{o \longleftrightarrow \partial \Lambda_n\}}$. Fix arbitrary orderings of $\Lambda$ and $\overline{E}(\Lambda)$.
    For $k \in \{1, \ldots, n\}$, define a decision tree $T(k)$ determining the function $f$ as follows.
    First explore all vertices in $\partial \Lambda_k$ one at a time according to the ordering, and set $V_1 = \partial \Lambda_k$.
    For $t \geq 1$, we determine the decision rule at time $|\partial \Lambda_k| + t$ inductively according to the following cases.
    \begin{itemize}
        \item If there exists an unexplored edge in $E(\Lambda_n) \coloneq \{xy \in E: x, y \in \Lambda_n\}$ with both endpoints explored and at least one endpoint in $V_t$, explore the first such edge $xy$. Set $V_{t+1} = V_t \cup \{x, y\}$ if the edge is open and $V_{t+1} = V_t$ otherwise ($V_t$ tracks which vertices are connected to $\partial \Lambda_k$).
        \item 
        If there is no such edge but there is an unexplored vertex in $\Lambda_n$ adjacent to a vertex in $V_t$, explore the first such vertex and set $V_{t+1} = V_t$.
        \item
        If no such vertex exists either, then we have finished exploring the cluster of $\partial \Lambda_k$, which means that the event $\{o \longleftrightarrow \partial \Lambda_n\}$ has been determined. We now explore the first unexplored vertex according to the arbitrary ordering, or if all vertices have been explored, we explore the first unexplored edge.
    \end{itemize}
    Since each edge is explored after its endpoints, $T(k)$ is admissible. Moreover, edge $xy$ is only revealed before time $\tau_{f, T(k)}$ if $xy \in E(\Lambda_n)$ and at least one of its endpoints is connected to $\partial \Lambda_k$, which implies 
    \begin{equation}
        \label{eq: edge revealment}
        \mu[R_{f, T(k)}(xy)] \leq \mathbbm{1}_{xy \in E(\Lambda_n)}(\mu[x \longleftrightarrow \partial \Lambda_k] + \mu[y \longleftrightarrow \partial \Lambda_k]).
    \end{equation}
    Also note that we only reveal a vertex $x$ before time $\tau_{f, T(k)}$ if $x \in \partial \Lambda_k$ or $x$ is both in $\Lambda_n$ and adjacent to a vertex in the cluster of $\partial \Lambda_k$. Hence,
    \begin{equation}
        \label{eq: vertex revealment}
        \mu[R_{f, T(k)}(x)] \leq \mathbbm{1}_{x \in \partial \Lambda_k} + \mathbbm{1}_{x \in \Lambda_n \setminus \partial \Lambda_k} \sum_{\substack{y \in \Lambda_n \\ y \sim x}} \mu[y \longleftrightarrow \partial \Lambda_k].
    \end{equation}
    Applying Theorem \ref{Thm: OSSS} to $T(k)$ and using \eqref{eq: edge revealment} and \eqref{eq: vertex revealment} yields
    \begin{equation}
    \label{eq: apply OSSS}
    \begin{aligned}
        \mathrm{Var}_{\mu}(f) \leq \frac{1}{\min_{x \in \Lambda_{n}} \mu[C_x]} \sum_{xy \in \overline{E}(\Lambda_n)} \bigg( \mu[x \longleftrightarrow \partial \Lambda_k] + \mu[y \longleftrightarrow \partial \Lambda_k] + \mathbbm{1}_{x \in \partial \Lambda_k} + \mathbbm{1}_{y \in \partial \Lambda_k}\\ + \sum_{\substack{z \in \Lambda_n \\ z \sim x}} \mu[z \longleftrightarrow \partial \Lambda_k] + \sum_{\substack{z \in \Lambda_n \\ z \sim y}} \mu[z \longleftrightarrow \partial \Lambda_k] \bigg) \mathrm{Cov}_{\mu}(f, \omega_{xy}).
    \end{aligned}
    \end{equation}
    Now observe that for every $x \in \Lambda_{n}$,
    \begin{equation*}
        \sum_{k=1}^{n} \mu[x \longleftrightarrow \partial \Lambda_k] \leq \sum_{k=1}^{n} \mu[x \longleftrightarrow \partial \Lambda_{|k-d_G(o, x)|}(x)]
        \leq 2 Q_n,
    \end{equation*}
    so summing \eqref{eq: apply OSSS} over $k \in \{1, \ldots, n\}$ gives
    \begin{align*}
        n\mathrm{Var}_{\mu}(f) \leq \frac{(4D + 6)Q_n}{\min_{x \in \Lambda_n} \mu[C_x]} \sum_{xy \in \overline{E}(\Lambda_n)} \mathrm{Cov}_{\mu}(f, \omega_{xy}),
    \end{align*}
    from which the desired result follows.
\end{proof}

We are now ready to prove Theorem \ref{Thm: sharpness RC}.
\begin{proof}[Proof of Theorem \ref{Thm: sharpness RC}]
    Let $M>0$ and define $\mathsf{p}_x = e^{M d_G(o,x)}$.
    By monotonicity in boundary conditions (Proposition \ref{prop: bc monotonicity for RC}), it suffices to consider the case of wired boundary conditions with $\sfb = \mathsf{p}$ for a sufficiently large choice of $M$.
    Set $\theta_n (\beta) = \Psi_{\Lambda_{2n}, \beta}^{(w, \mathsf{p})}[o \longleftrightarrow \partial \Lambda_n]$.
    Combining Proposition \ref{prop: beta derivative} and Lemma~\ref{lem:cov comparison new}, we see that for any $n \geq 1$ and $\beta > 0$
    \begin{equation*}
        \frac{\mathrm{d} \theta_n}{\mathrm{d} \beta} \geq \varepsilon \sum_{e \in \overline{E}(\Lambda_{2n})} \mathrm{Cov}_{\Lambda_{2n}, \beta}^{(w, \mathsf{p})}(\mathbbm{1}_{\{o \longleftrightarrow \partial \Lambda_n\}}, \omega_e).
    \end{equation*}
    Observe that $\Psi_{\Lambda_{2n}, \beta}^{(w, \mathsf{p})}$ satisfies weak monotonicity and \eqref{eq: zero conditioning} by Proposition~\ref{prop: weak monotonicity} and Lemma~\ref{Lemma: zero conditioning}, respectively. Hence we can apply Lemma \ref{Lem: decision tree} to get
    \begin{equation*}
        \frac{\mathrm{d} \theta_n}{\mathrm{d} \beta} \geq \frac{\varepsilon n \min_{x \in \Lambda_n} \Psi_{\Lambda_{2n}, \beta}^{(w, \mathsf{p})}[C_x]}{(4D+6) Q_n} \theta_n (1 - \theta_n ).
    \end{equation*}
    
    Recall that $\beta_c>0$ is non-trivial by Proposition~\ref{prop:non-trivial} and consider some $\beta_0\in (0,\beta_c)$. Since $\beta_0$ is arbitrary, to get the desired subcritical sharpness, it suffices to prove that $\theta_n(\beta)$ decays exponentially for every $\beta\in [\beta_0, 2\beta_0]\cap (0,\beta_c)$. To this end, our aim is to show that there exists $t>0$ such that for every $\beta\in [\beta_0,2\beta_0]$,
    \begin{equation}
    \label{eq: differential inequality}
    \frac{\mathrm{d} \theta_n}{\mathrm{d} \beta} \geq t \frac{n}{\sum_{k = 0}^{n-1} \theta_{k}(\beta)}\theta_n(\beta).
    \end{equation}
    
First, note that for $\beta\in [\beta_0, 2\beta_0]$,  $\varepsilon$ is bounded away from $0$ since it depends continuously on $\beta$. Next, we handle $\Psi_{\Lambda_{2n}, \beta}^{(w, \mathsf{p})}[C_x](1 - \theta_n(\beta))$. Since the distance between $\Lambda_n$ and the boundary of $\Lambda_{2n}$ is equal to $n$ and the boundary conditions $\mathsf{p}$ have growth that is at most exponential, we can use the regularity statement of Proposition~\ref{prop: regularity} to get that for every $M>0$ there exists $C>0$ such that for every $n\geq 1$, every $x\in \Lambda_n$ and every $\beta\in [\beta_0,2\beta_0]$, 
\[\Psi_{\Lambda_{2n}, \beta}^{(w, \mathsf{p})}[\sfa_x\leq C, \sfa_y\leq C \, \forall\, y\sim x]\geq 1/2.
\]
Hence 
\begin{equation}\label{eq:C_x bound}
\Psi_{\Lambda_{2n}, \beta}^{(w, \mathsf{p})}[C_x]\geq \frac{e^{-2\beta D C^2}}{2}\geq \frac{e^{-4\beta_0 D C^2}}{2},
\end{equation}
and similarly 
\begin{equation}\label{eq:1-theta bound}
1-\theta_n(\beta)\geq \Psi_{\Lambda_{2n}, \beta}^{(w, \mathsf{p})}[C_o]\geq \frac{e^{-4\beta_0 D C^2}}{2}.
\end{equation}

To bound $Q_n$, consider $x\in \Lambda_n$ and let $\tau_x \mathsf{p}$ be the field $\mathsf{p}$ translated by $x$, i.e.\ $(\tau_x \mathsf{p})_y = e^{M d_G(x, y)}$. By inclusion of events and monotonicity in boundary conditions,
\begin{equation}\label{eq:Q_n bound}
\begin{aligned}
\sum_{k=0}^{n-1} \Psi_{\Lambda_{2n}, \beta}^{(w, \mathsf{p})}[x \longleftrightarrow \partial \Lambda_k(x)]&\leq 4 \sum_{k\leq n/4} \Psi_{\Lambda_{2n}, \beta}^{(w, \mathsf{p})}[x \longleftrightarrow \partial \Lambda_k(x)]\\ 
&\leq 4\sum_{k\leq n/4} \theta_k(\beta) + \Psi_{\Lambda_{2n}, \beta}^{(w, \mathsf{p})}[\mathcal{E}_k(x)],
\end{aligned}
\end{equation}
where $\mathcal{E}_k(x)=\{\exists \, y\in \partial^{\mathrm{ext}} \Lambda_{2k}(x): \sfa_y > (\tau_x \mathsf{p})_y\}$ and we have used that $\Psi_{\Lambda_{2k}(x), \beta}^{(w, \tau_x \mathsf{p})}[x \longleftrightarrow \partial \Lambda_k(x)] = \theta_k(\beta)$ by translation invariance. For $k \leq n/4$, $\Lambda_{2k}(x) \subset \Lambda_{3n/2}$, so by regularity there exists $c>0$ such that for every $n \geq 1$, $k\leq n/4$ and $x\in \Lambda_n$,
\begin{equation}\label{eq:E_k bound}
\Psi_{\Lambda_{2n}, \beta}^{(w, \mathsf{p})}[\mathcal{E}_k(x)]\leq e^{-ck}.
\end{equation}
On the other hand, by simply opening a geodesic connecting $o$ to $\partial \Lambda_k$ and using the FKG inequality, we see that $\Psi_{\Lambda_{2k}, \beta}^{(w, \mathsf{p})}[o \longleftrightarrow \partial \Lambda_k]\geq e^{-c'k}$, where here we note that the probability an edge is open remains bounded away from $0$ by monotonicity in the boundary conditions and the fact that this holds for $\Lambda$ consisting of 2 neighbouring vertices. By tuning the value of $M$, we can make the value of $c$ large enough so that $c>c'$. We now fix such a value of $M$ and for this choice we get 
\[
\Psi_{\Lambda_{2n}, \beta}^{(w, \mathsf{p})}[\mathcal{E}_k(x)]\leq \Psi_{\Lambda_{2k}, \beta}^{(w, \mathsf{p})}[o \longleftrightarrow \partial \Lambda_k].
\]
Combining the latter inequality with \eqref{eq:Q_n bound}, we deduce that 
\[
Q_n \leq 8 \sum_{k\leq n/4} \theta_k(\beta) \leq 8 \sum_{k=0}^{n-1} \theta_k(\beta),
\]
which when combined in turn with \eqref{eq:C_x bound} and \eqref{eq:1-theta bound} gives \eqref{eq: differential inequality}.

Note that $\theta_n(\beta)$ converges to $\Psi^{1}_{\beta}[o\longleftrightarrow \infty]$ as $n\to\infty$. Indeed, recall that $\Psi_{\Lambda_{2n}, \beta}^{(w, \mathsf{p})}$ converges weakly to $\Psi_{\beta}^{1}$ by Proposition~\ref{prop: plus convergence}. Now 
\[
\Psi^{1}_{\beta}[o\longleftrightarrow \infty]\leq \liminf_{n\to\infty}\theta_n(\beta)+\liminf_{n\to\infty}\Psi^{1}_{\beta}[\mathcal{E}_n(o)]=\liminf_{n\to\infty}\theta_n(\beta)
\]
by monotonicity in boundary conditions and \eqref{eq:E_k bound}, and for every $k\geq 1$,
\[
\limsup_{n\to\infty}\theta_n(\beta)\leq \lim_{n\to\infty}\Psi_{\Lambda_{2n}, \beta}^{(w, \mathsf{p})}[o\longleftrightarrow \partial \Lambda_k]=\Psi_{\beta}^{1}[o\longleftrightarrow \partial \Lambda_k]
\xrightarrow[k\to\infty]{}
\Psi_{\beta}^{1}[o\longleftrightarrow \infty].
\]
Applying \cite[Lemma 3.1]{sharpness_annals} to $f_n \coloneq \theta_n/t$, we obtain the existence of a critical point $\beta_1 \in [\beta_0, 2\beta_0]$ satisfying the following two properties:
\begin{itemize}
\item
For any $\beta \in [\beta_0, \beta_1)$, there exists $c = c_{\beta} > 0$ such that for any $n\geq 0$, $\theta_n(\beta) \leq e^{-c n}$.
\item
For any $\beta \in (\beta_1, 2\beta_0]$, $\Psi_{\beta}^{1}[o \longleftrightarrow \infty] \geq t(\beta - \beta_1)>0$.
\end{itemize}
Since $\beta_c = \inf\{\beta \geq 0 : \Psi_{\beta}^{1} [o \longleftrightarrow \infty] > 0\}$ (see Remark \ref{remark: infinite coupling}), it follows that $\theta_n(\beta)$ decays exponentially for every $\beta\in [\beta_0, 2\beta_0]\cap (0,\beta_c)$, and we also get the mean-field lower bound. Since $\Psi_{\Lambda_{2n}, \beta}^{(w, \mathsf{p})}[o \longleftrightarrow \partial \Lambda_{2n}] \leq \theta_n$, we deduce that $\Psi_{\Lambda_{m}, \beta}^{(w, \mathsf{p})}[o \longleftrightarrow \partial \Lambda_{m}]$ decays exponentially for even $m$, and one can handle the case of odd $m$ similarly. 
\end{proof}

\section{Proof of Theorem \ref{thm: main theorem}}\label{sec:sharpness}
In this section, we use Theorem \ref{Thm: sharpness RC} to prove Theorem \ref{thm: main theorem}. 

\begin{proof}[Proof of Theorem~\ref{thm: main theorem}]
Let $\beta<\beta_c$, $x\in V$ with $d_G(o, x) = k \geq 1$, and $n \geq k$. Let $\lambda > 0$ be a constant to be determined and define $\mathsf{p} \in (\mathbb{R}^{+})^{V}$ by $\mathsf{p}_y = e^{\lambda d_G(o,y)}$.
By Proposition~\ref{prop:monotonicity in bc} and Corollary \ref{Cor: coupling}, for any boundary conditions $|\eta| \leq \mathsf{p}$ we have
\[
\langle \varphi_o \varphi_x \rangle^{\eta}_{\Lambda_n,\beta}\leq
\langle \varphi_o \varphi_x \rangle^{\mathsf{p}}_{\Lambda_n,\beta}=\Psi_{\Lambda_{n}, \beta}^{(w, \mathsf{p})}[\sfa_o \sfa_x \mathbbm{1}_{o\longleftrightarrow x}].
\]
Recall that the distribution of $\sfa$ under $\Psi_{\Lambda_n, \beta}^{(w, \mathsf{p})}$ is the same as the distribution of $|\varphi|$ under $\nu_{\Lambda_n, \beta}^{\mathsf{p}}$,
so we can use the Cauchy-Schwarz inequality to get 
\[
\langle \varphi_o \varphi_x \rangle^{\eta}_{\Lambda_n, \beta}\leq \sqrt{\langle \varphi^2_o \varphi^2_x \rangle^{\mathsf{p}}_{\Lambda_n, \beta} \Psi_{\Lambda_{n}, \beta}^{(w, \mathsf{p})}[o\longleftrightarrow x]}.
\]
We now use regularity to bound $\langle \varphi^2_o \varphi^2_x \rangle^{\mathsf{p}}_{\Lambda_n, \beta}$. Fix $a > 2 D \beta$ and let $A_x = A(x, \Lambda_n, \mathsf{p}, \beta, a)$. Using Proposition \ref{prop: regularity}, we have for some constant $B > 0$ 
\[\langle \varphi^2_o \varphi^2_x \rangle^{\mathsf{p}}_{\Lambda_n, \beta}\leq A_x^2\langle \varphi^2_o\rangle^{\mathsf{p}}_{\Lambda_n, \beta} + \langle \varphi^2_o\varphi^2_x \mathbbm{1}_{|\varphi_x|\geq  A_x}\rangle^{\mathsf{p}}_{\Lambda_n, \beta} 
\leq A_x^2\langle \varphi^2_o\rangle^{\mathsf{p}}_{\Lambda_n, \beta} + e^{2B A_x^2} \langle \varphi^2_o\varphi^2_x \mathbbm{1}_{|\varphi_x|\geq  A_x}\rangle^{0}_{\{o, x\}, 0, \rho_a},\]
where we have used that $A_o \leq A_x$, and we note that $\langle \cdot \rangle^{0}_{\{o, x\}, 0, \rho_a}$ is a product measure. For any $C > 0$, we have 
\begin{equation*}
\nu_{\{x\}, 0, \rho_a}^{0}[|\varphi_x| \geq A_x] \leq e^{-C A_x^2}\nu_{\{x\}, 0, \rho_a}^{0}[e^{C|\varphi_x|^2}] 
\end{equation*}
by Markov's inequality applied to the random variable $e^{C|\varphi_x|^2}$,
and choosing $C$ appropriately we can bound $e^{2B A_x^2} \langle \varphi^2_o\varphi^2_x \mathbbm{1}_{|\varphi_x|\geq  A_x}\rangle^{0}_{\{o, x\}, 0, \rho_a}$ by a constant.
We hence obtain for some constant $C' > 0$
\begin{equation}
    \label{eq: proof of sharpness}
    \langle \varphi_o \varphi_x \rangle^{\eta}_{\Lambda_n, \beta}\leq C' e^{\lambda k} \sqrt{ \Psi_{\Lambda_{n}, \beta}^{(w, \mathsf{p})}[o\longleftrightarrow x]}.
\end{equation}

We next bound $\Psi_{\Lambda_{n}, \beta}^{(w, \mathsf{p})}[o\longleftrightarrow x]$.
Let $\mathsf{q}_y = e^{M d_G(o,y)}$ for some fixed $M > \lambda$ and
define the event $\mathcal{E}_k=\{\exists \, y\in \partial^{\mathrm{ext}} \Lambda_{k}: \sfa_y > \mathsf{q}_y\}$.
By the domain Markov property and monotonicity in boundary conditions, we have
\begin{equation*}
    \Psi_{\Lambda_{n}, \beta}^{(w, \mathsf{p})}[o\longleftrightarrow x] \leq
    \Psi_{\Lambda_{k}, \beta}^{(w, \mathsf{q})}[o\longleftrightarrow x] + \Psi_{\Lambda_n, \beta}^{(w, \mathsf{p})}[\mathcal{E}_k].
\end{equation*}
The connection probability $\Psi_{\Lambda_{k}, \beta}^{(w, \mathsf{q})}[o\longleftrightarrow x]$ decays exponentially in $k$ by Theorem~\ref{Thm: sharpness RC} and inclusion of events, while $\Psi_{\Lambda_n, \beta}^{(w, \mathsf{p})}[\mathcal{E}_k] \leq \Psi_{\Lambda_n, \beta}^{(w, \mathsf{q})}[\mathcal{E}_k]$ decays at least exponentially in $k$ by regularity. Combining this with \eqref{eq: proof of sharpness}, we have that $\langle \varphi_o \varphi_x \rangle^{\eta}_{\Lambda_n,\beta}$ decays exponentially in $k$ for any $|\eta|\leq \mathsf{p}$ so long as $\lambda$ is small enough, and by sending $n$ to infinity and using Proposition~\ref{Prop: plus measure spin model} (ii) we obtain that $\langle \varphi_o \varphi_x \rangle^{+}_{\beta}$ decays exponentially in $k$ as well.
\end{proof}

\bibliographystyle{plain}
\bibliography{references}

@article{random_tangled,
author = {T.~S.~Gunaratnam and C.~Panagiotis and R.~Panis and F.~Severo},
title = {Random tangled currents for $\varphi^{4}$: Translation invariant {G}ibbs measures and continuity of the phase transition},
journal={Journal of the European Mathematical Society (to appear)},
year={2025}
}

@article{well_behaved,
author = {Gunaratnam, T. S. and Panagiotis, C. and Panis, R. and Severo, F.},
journal = {arXiv:2501.05353},
year = {2025},
title = {The supercritical phase of the $\varphi^4$ model is well behaved}
}

@article{PV26,
author = {Panagiotis, C. and Veitch, W.},
journal = {arXiv:2603.26319},
year = {2026},
title = {Regularity of {G}ibbs measures for unbounded spin systems on general graphs}
}

@article{sharpness_annals,
author = {Duminil-Copin, H. and Raoufi, A. and Tassion, V.},
title = {Sharp phase transition for the random-cluster and {P}otts models via decision trees},
journal = {Annals of Mathematics},
year = {2019},
volume = {189},
number = {1},
pages = {75--99}
}

@INPROCEEDINGS{OSSS,
  author={O'Donnell, R. and Saks, M. and Schramm, O. and Servedio, R. A.},
  booktitle={46th Annual IEEE Symposium on Foundations of Computer Science (FOCS'05)}, 
  title={Every decision tree has an influential variable}, 
  year={2005},
  volume={},
  number={},
  pages={31-39},
  doi={10.1109/SFCS.2005.34}}

@article{BL,
volume = {141},
year = {1966},
pages = {517-524},
title = {Theory of the first-order magnetic phase change in {UO}$_2$},
author = {Blume, M.},
copyright = {Copyright 2007 Elsevier B.V., All rights reserved.},
issn = {0031-899X},
journal = {Physical review},
language = {eng},
number = {2},
}

@article{Capel,
volume = {32},
year = {1966},
pages = {966-988},
publisher = {Elsevier B.V},
title = {On the possibility of first-order phase transitions in {I}sing systems of triplet ions with zero-field splitting},
author = {Capel, H. W.},
copyright = {1957},
issn = {0031-8914},
journal = {Physica},
language = {eng},
number = {5},
}

@article{Blume-Capel,
author = {Gunaratnam, T. S. and Krachun, D. and Panagiotis, C.},
issn = {2690-0998},
journal = {Probability and Mathematical Physics},
language = {eng},
number = {3},
pages = {785-845},
title = {Existence of a tricritical point for the {B}lume–{C}apel model on $\mathbb{Z}^{d}$},
volume = {5},
year = {2024},
}

@incollection{DC_lecturenotes,
author = {H.~Duminil-Copin},
note = {Available at https://arxiv.org/abs/1707.00520},
title = {Lectures on the {I}sing and {P}otts models on the hypercubic lattice},
booktitle = {PIMS-CRM Summer School in Probability},
pages = {35--161},
year = {2017},
publisher = {Springer},
}

@article{LSS97,
  title={Domination by product measures},
  author={T.~M.~Liggett and R.~H.~Schonmann and A.~M.~Stacey},
  journal={The Annals of Probability},
  volume={25},
  number={1},
  pages={71--95},
  year={1997},
  publisher={Institute of Mathematical Statistics}
}

@article{ABF87,
  title={The phase transition in a general class of {I}sing-type models is sharp},
  author={M.~Aizenman and D.~J.~Barsky and R.~Fern{\'a}ndez},
  journal={Journal of Statistical Physics},
  volume={47},
  pages={343--374},
  year={1987},
  publisher={Springer}
}

@book{friedli_velenik_2017,
title = {{Statistical Mechanics of Lattice Systems: A Concrete Mathematical Introduction}},
ISBN = {978-1-107-18482-4},
publisher = {Cambridge University Press},
author = {Friedli, S. and Velenik, Y.},
year = {2017}
}

@book{Glimm-Jaffe,
publisher = {Springer New York},
title = {Quantum Physics : A Functional Integral Point of View },
author = {Glimm, J. and Jaffe, A.},
booktitle = {Quantum Physics : A Functional Integral Point of View},
edition = {2nd},
isbn = {1-4612-4728-4},
language = {eng},
year = {1987},
}

@book{WP21,
  author    = {W.~Werner and E.~Powell},
  title     = {Lecture Notes on the Gaussian Free Field},
  series    = {Cours Sp{\'e}cialis{\'e}s},
  volume    = {28},
  publisher = {Soci{\'e}t{\'e} Math{\'e}matique de France},
  address   = {Paris},
  year      = {2021}
}

@article{DuminilTassionNewProofSharpness2016,
  title={A new proof of the sharpness of the phase transition for {B}ernoulli percolation and the {I}sing model},
  author={Duminil-Copin, H. and Tassion, V.},
  journal={Communications in Mathematical Physics},
  volume={343},
  pages={725--745},
  year={2016},
  publisher={Springer}
}

@inproceedings{mensikov1986coincidence,
  title={Coincidence of critical points in percolation problems},
  author={M.~V.~Mensikov},
  booktitle={Soviet Mathematics Doklady},
  volume={33},
  pages={856--859},
  year={1986}
}

@article{aizenman1987sharpness,
  title={Sharpness of the phase transition in percolation models},
  author={M.~Aizenman and D.~J.~Barsky},
  journal={Communications in Mathematical Physics},
  volume={108},
  number={3},
  pages={489--526},
  year={1987},
  publisher={Springer}
}

@article{Van23,
  title={{Exponential decay of the volume for Bernoulli percolation: a proof via stochastic comparison}},
  author={Vanneuville, H.},
  journal={Annales Henri Lebesgue},
  pages={101--112},
  volume={8},
  year={2025}
}

@book{Grimmett_RC,
isbn = {3540328904},
issn = {0072-7830},
language = {eng},
publisher = {Springer-Verlag},
series = {Grundlehren der mathematischen Wissenschaften},
title = {The Random-Cluster Model},
volume = {333},
year = {2006},
author = {Grimmett, G. R.},
}

@article{GG,
author = {B.~T.~Graham and G.~R.~Grimmett},
title = {{Random-cluster representation of the Blume--Capel model}},
journal = {Journal of Statistical Physics},
volume={125},
number={2},
pages={283--316},
year = {2006}
}

@article{Edwards-Sokal,
  title = {Generalization of the {F}ortuin-{K}asteleyn-{S}wendsen-{W}ang representation and {M}onte {C}arlo algorithm},
  author = {Edwards, R. G. and Sokal, A. D.},
  journal = {Physical Review D},
  volume = {38},
  number = {6},
  pages = {2009--2012},
  year = {1988},
  publisher = {American Physical Society},
  doi = {10.1103/PhysRevD.38.2009},
  url = {https://link.aps.org/doi/10.1103/PhysRevD.38.2009}
}

@article{Lebowitz_Presutti,
issn = {0010-3616},
journal = {Communications in Mathematical Physics},
language = {eng},
number = {3},
pages = {195-218},
title = {Statistical mechanics of systems of unbounded spins},
volume = {50},
year = {1976},
author = {J.~L.~Lebowitz and E.~Presutti},
}

@article{P36,
  title={{On Ising's model of ferromagnetism}},
  author={R.~Peierls},
  journal={Mathematical Proceedings of the Cambridge Philosophical Society},
  volume={32},
  number={3},
  pages={477--481},
  year={1936}
}

@article{GJS75,
  title={Phase transitions for $\varphi^4_2$ quantum fields},
  author={J.~Glimm and A.~Jaffe and T.~Spencer},
  journal={Communications in Mathematical Physics},
  volume={45},
  pages={203--216},
  year={1975}
}

@article{Ruelle_estimates,
issn = {0010-3616},
journal = {Communications in Mathematical Physics},
language = {eng},
number = {3},
pages = {189-194},
publisher = {Springer-Verlag},
title = {Probability estimates for continuous spin systems},
volume = {50},
year = {1976},
author = {Ruelle, D.},
}

@article{Ruelle1970,
volume = {18},
year = {1970},
pages = {127-159},
publisher = {Springer-Verlag},
author = {Ruelle, D.},
copyright = {Copyright 2007 Elsevier B.V., All rights reserved.},
issn = {0010-3616},
journal = {Communications in Mathematical Physics},
language = {eng},
number = {2},
title = {Superstable interactions in classical statistical mechanics},
}

@article{Lenz1920beitrag,
  title={{Beitrag zum Verst{\"a}ndnis der magnetischen Erscheinungen in festen K{\"o}rpern}},
  author={Lenz, W.},
  journal={Z. Phys.},
  volume={21},
  pages={613--615},
  year={1920}
}

@phdthesis{Ising1924,
  title={Beitrag zur theorie des ferro-und paramagnetismus},
  author={Ising, E.},
  year={1924},
  school={Grefe \& Tiedemann Hamburg, Germany}
}

@inproceedings{N66,
  title={A quartic interaction in two dimensions},
  author={E.~Nelson},
  booktitle={Mathematical Theory of Elementary Particles, (Dedham, Mass., 1965)},
  pages={69--73},
  year={1966}
}

@article{GJ73,
  title={Positivity of the $\varphi^4_3$ Hamiltonian},
  author={J.~Glimm and A.~Jaffe},
  journal={Fortschritte der Physik},
  volume={21},
  number={7},
  pages={327--376},
  year={1973}
}

@article{frohlich1977pure,
  title={{Pure states for general $P(\varphi)_2$ theories: construction, regularity and variational equality}},
  author={J.~Fr\"ohlich and B.~Simon},
  journal={Annals of Mathematics},
  pages={493--526},
  year={1977},
  publisher={JSTOR}
}

@article{borgs1989unified,
  title={A unified approach to phase diagrams in field theory and statistical mechanics},
  author={C.~Borgs and J.~Z.~Imbrie},
  journal={Communications in Mathematical Physics},
  volume={123},
  number={2},
  pages={305--328},
  year={1989},
  publisher={Springer}
}

@article{DGRSY20,
author = {H.~Duminil-Copin and S.~Goswami and A.~Raoufi and F.~Severo and A.~Yadin},
title = {{Existence of phase transition for percolation using the Gaussian free field}},
volume = {169},
journal = {Duke Mathematical Journal},
number = {18},
publisher = {Duke University Press},
pages = {3539 -- 3563},
year = {2020}
}

@article{EST25,
  title={Counting minimal cutsets and $p_c<1$},
  author={Easo, P. and Severo, F. and Tassion, V.},
  journal={Forum of Mathematics, Pi},
  volume={13},
  year={2025},
  organization={Cambridge University Press}
}

@article{FrohlichSimonSpencerIRBounds1976,
  title={Infrared bounds, phase transitions and continuous symmetry breaking},
  author={Fr{\"o}hlich, J. and Simon, B. and Spencer, T},
  journal={Communications in Mathematical Physics},
  volume={50},
  number={1},
  pages={79--95},
  year={1976},
  publisher={Springer}
}

@article{Lebowitz1977CoexistencePhasesIsing,
  title={Coexistence of phases in {I}sing ferromagnets},
  author={Lebowitz, J. L.},
  journal={Journal of Statistical Physics},
  volume={16},
  number={6},
  pages={463--476},
  year={1977}
}

@article{LammersOtt2021,
  title={{Delocalisation and absolute-value-FKG in the solid-on-solid model}},
  author={Lammers, P. and Ott, S.},
  journal={Probability Theory and Related Fields},
  volume={188},
  pages={63--87},
  year={2024}
}

@article{GS,
  title={{The $(\varphi^4)_2$ field theory as a classical Ising model}},
  author={R.~B.~Griffiths and B.~Simon},
  journal={Communications in Mathematical Physics},
  volume={33},
  number={2},
  pages={145--164},
  year={1973}
}

@article{Hut20,
  title={New critical exponent inequalities for percolation and the random cluster model},
  author={Hutchcroft, T.},
  journal={Probability and Mathematical Physics},
  volume={1},
  number={1},
  pages={147--165},
  year={2020}
}

\end{document}